\documentclass[11pt,reqno]{amsart}
\usepackage[paperheight=279mm,paperwidth=18cm,textheight=26cm,textwidth=14cm,includehead]{geometry}
\usepackage[utf8]{inputenc}
\usepackage[unicode]{hyperref}
\hypersetup{
bookmarks=true,
colorlinks=true,
citecolor=[rgb]{0,0,0.5},
linkcolor=[rgb]{0,0,0.5},
urlcolor=[rgb]{0,0,0.75},
pdfpagemode=UseNone,
pdfstartview=FitH,
pdfdisplaydoctitle=true,
pdflang=en-US
}

\usepackage{amssymb}
\usepackage{amsmath}
\usepackage{amsthm}
\usepackage{amsfonts}
\usepackage{mathtools}
\usepackage{graphicx}
\usepackage{xcolor}

\numberwithin{equation}{section}
\newtheorem{theorem}{Theorem}[section]

\newtheorem{corollary}[theorem]{Corollary}
\newtheorem{lemma}[theorem]{Lemma}
\newtheorem{conjecture}[theorem]{Conjecture}
\theoremstyle{definition}

\newtheorem{definition}[theorem]{Definition}
\newtheorem{remark}[theorem]{Remark}

\DeclareMathOperator{\supp}{supp}

\DeclarePairedDelimiter\abs{\lvert}{\rvert}
\DeclarePairedDelimiter\norm{\lVert}{\rVert}

\DeclarePairedDelimiterX\innerp[2]{\langle}{\rangle}{#1,#2}
\providecommand\given{}
\newcommand\SetSymbol[1][]{%
\nonscript\:#1\vert
\allowbreak
\nonscript\:
\mathopen{}}
\DeclarePairedDelimiterX\Set[1]\{\}{%
\renewcommand\given{\SetSymbol[\delimsize]}
#1
}

\newcommand{\La}{\langle}
\newcommand{\Ra}{\rangle}

\newcommand{\pd}{\partial}

\def\supp{\operatorname{supp}}

\newcommand{\ch}{\operatorname{ch}}

\newcommand{\al}{\alpha}
\newcommand{\eps}{\varepsilon}
\newcommand{\la}{\lambda}
\newcommand{\vf}{\varphi}
\newcommand{\om}{\omega}
\newcommand{\Om}{\Omega}

\newcommand{\cD}{\mathcal D}
\newcommand{\cE}{\mathcal E}
\newcommand{\cF}{\mathcal F}

\newcommand{\cP}{\mathcal P}

\newcommand{\bC}{\mathbb C}
\newcommand{\bD}{\mathbb D}
\newcommand{\bE}{\mathbb E}

\newcommand{\bI}{\mathbb I}
\newcommand{\bP}{\mathbb P}
\newcommand{\bR}{\mathbb R}
\newcommand{\bT}{\mathbb T}
\newcommand{\bV}{\mathbb V}

\newcommand{\bfI}{\mathbf{I}}
\newcommand{\bfV}{\mathbf{V}}

\newcommand{\1}{{\mathbf 1}}

\newcommand{\one}{\mathbf{1}}
\newcommand{\dif}{\mathop{}\!\mathrm{d}} 

\def\PZdefchar#1{
  \expandafter\def\csname frak#1\endcsname{\mathfrak{#1}}
  \expandafter\def\csname bb#1\endcsname{\mathbb{#1}}
  \expandafter\def\csname bf#1\endcsname{\mathbf{#1}}
  \expandafter\def\csname scr#1\endcsname{\mathcal{#1}}
  \expandafter\def\csname cal#1\endcsname{\mathcal{#1}}}
\def\PZdefloop#1{\ifx#1\PZdefloop\else\PZdefchar#1\expandafter\PZdefloop\fi}
\PZdefloop abcdefghijklmnopqrstuvwxyzABCDEFGHIJKLMNOPQRSTUVWXYZ\PZdefloop

\newcommand{\downset}{{\downarrow}}
\newcommand{\upset}{{\uparrow}}

\newcommand{\hypothesistensor}[1]{Let $w : T^{#1} \to [0,\infty)$ be of tensor product form.}
\hypersetup{
pdftitle={Tri-tree Carleson embeddings},
pdfauthor={Mozolyako, Volberg, Zorin-Kranich},
}

\begin{document}

\title[Carleson embedding on tri-tree and on tri-disc]%
{Carleson embedding on tri-tree and on tri-disc}
\author[P.~Mozolyako]{Pavel Mozolyako}
\thanks{PM is supported by the Russian Science Foundation grant 17-11-01064}
\address[P.~Mozolyako]{Universit\`{a} di Bologna, Department of Mathematics, Piazza di Porta S. Donato, 40126 Bologna (BO)}
\email{pavel.mozolyako@unibo.it}
\author[G.~Psaromiligkos]{Georgios Psaromiligkos}
\address[G.~Psaromiligkos]{Department of Mathematics, Michigan Sate University, East Lansing, MI. 48823}
\email{psaromil@math.msu.edu}
\author[A.~Volberg]{Alexander Volberg}
\thanks{AV is partially supported by the NSF grant DMS-160065 and DMS 1900268 and by Alexander von Humboldt foundation}
\address[A.~Volberg]{Department of Mathematics, Michigan Sate University, East Lansing, MI. 48823}
\email{volberg@math.msu.edu}
\author[P.~Zorin-Kranich]{Pavel Zorin-Kranich}
\thanks{PZ was partially supported by the Hausdorff Center for Mathematics (DFG EXC 2047)}
\address[P.~Zorin-Kranich]{Mathematical Institute, University of Bonn, Bonn, Germany}
\email{pzorin@uni-bonn.de}
\begin{abstract}
We prove multi-parameter dyadic embedding theorem for Hardy operator on the multi-tree. We also show that for a large class of Dirichlet spaces in bi-disc and tri-disc this proves the embedding theorem of those Dirichlet spaces of holomorphic function on bi- and tri-disc. We completely describe the Carleson measures for such embeddings. The result below generalizes embedding result of \cite{AMPVZ} from bi-tree to tri-tree. One of our embedding description is similar to Carleson--Chang--Fefferman condition and involves dyadic open sets. On the other hand, the unusual feature of \cite{AMPVZ} was that
embedding on bi-tree turned out to be equivalent to one box Carleson condition. This is in striking difference to works of Chang--Fefferman and well known Carleson quilt counterexample. We prove here the same unexpected result for the tri-tree and tri-disc. Finally, we explain the obstacle that prevents us from proving our results on polydiscs of dimension four and higher.
\end{abstract}
\maketitle

\section{Introduction and the main result}
\label{intro}

The present article treats a two weight problem about multi-parameter paraproduct operators.
Singular bi-parameter and multi-parameter operators enjoyed and continue to enjoy much attention, see \cite{RF,RF1,RF2}, \cite{P}, \cite{JLJ}, \cite{JLJ2} \cite{BP}.
They are notoriously difficult.
Two weight problems for singular integrals were studied in a series of papers by Nazarov, Treil, and Volberg on dyadic singular operators and in a series of papers by Lacey, Shen, Sawyer, and Uriarte-Tuero on the Hilbert transform, see \cite{NTV99}, \cite{NTV08}, \cite{LSSUT}, \cite{La}, and the references therein.
Another example is a very recent paper by Iosevich, Krause, Sawyer, Taylor, and Uriarte-Tuero \cite{IKSTUT} on the two weight problem for the spherical maximal operator motivated by Falconer's distance set problem.

Classically, an estimate of paraproduct tri-linear forms \cite{GT} is based on $T1$ theorem of David and Journ\'e. The theory of Carleson measures (or classical $BMO$ theory) is involved.
It is well known \cite{ChF1, ChF2, JLJ, JLJ2} that in the multi-parameter setting all these results and concepts of Carleson measure, $BMO$, John--Nirenberg inequality, Calder\'on--Zygmund decomposition are much more delicate. Paper \cite{MPTT1} develops a completely new approach to prove natural tri-linear bi-parameter estimates on bi-parameter paraproducts, especially outside of Banach range. In \cite{MPTT1} Journ\'e's lemma \cite{JLJ2} was used, but the approach did not generalize to multi-parameter paraproduct forms. This issue was resolved in \cite{MPTT2}, where a simplified method was used to address the multi-parameter paraproducts. 

We consider here bi-parameter and tri-parameter paraproducts and reveal the obstacle to treat the dimension $4$ objects.  Our paraproducts are only dyadic ones, and we estimate them only in $L^2$. But we consider a two weight problem. One weight is arbitrary and the other one is dictated by the problem from complex analysis in the polydisc (our original motivation). This other weight has the product structure because of this original motivation. We are able to give the necessary and sufficient condition for the two weight boundedness of such multi-parameter  paraproducts in two and three parameter case (and of course in one parameter case).

Three remarks are in order: a) the general two weight problem  even for two parameter paraproducts seems to not having a simple necessary and sufficient criterion at all (unlike a one parameter case of dyadic paraproducts, whose solution is basically due to Eric Sawyer); so it is a ``miracle'' that the full solution exists when one measure is arbitrary, and another one has a product structure; b) this full solutions continues to amaze us because it seemingly goes against a famous Carleson counterexample in the theory of Chang--Fefferman product $BMO$; c) it is also amazing that problem about holomorphic functions in the polydisc can be reduced to dyadic problems having nothing to do with complex analysis, the information--in many cases--is not getting lost.

\medskip

\noindent{\bf Acknowledgement}. We are grateful to Sergei Treil for indicating a faulty reasoning in Section 2.4, we corrected this reasoning.


\subsection{Background. Embedding from $L^2(m_2)$ to $\ell^2(T^2, \{\beta^2\})$}
\label{sec:background}

Lennart Carleson showed in \cite{Car} that the natural generalization, using a ``box'' condition, from the one parameter case (disc) to the bi-parameter case (bi-disc) of his embedding theorem does not work.
Sun-Yang A.~Chang in \cite{Ch} found the necessary and sufficient condition for the validity of the Carleson embedding for bi-harmonic extensions into the bi-disc.

The discrete versions of these results can be motivated by considering a bi-parameter dyadic paraproduct.
For a dyadic rectangle $R = I \times J \subseteq [0,1]^{2}$ denote by $h_{R}(x,y) = h_{I}(x) h_{J}(y)$ an associated $L^{2}$ normalized Haar function.
The simplest example of a \emph{bi-parameter dyadic paraproduct} is the operator
\[
\Pi_b \vf := \sum_{R}\La \vf\Ra_R (b, h_R) h_R\,.
\]
The paraproduct $\Pi_{b}$ is a bounded operator on $L^{2}$ with respect to the Lebesgue measure $m$ on $[0,1]^{2}$ if and only if we have
\begin{equation}
\label{para2}
\sum_{R}\La \vf\Ra_R^2\, \beta_R^2 \le C \int\vf^2 dm_2,
\end{equation}
where $\beta_R := (b, h_R)$ are Haar coefficients of the function $b$.
In analogy to the one-parameter Carleson embedding one could ask whether \eqref{para2} is equivalent to the ``box'' condition
\begin{equation}
\label{box1}
\sum_{R\subseteq R_0} \beta_R^2\le C' m_2(R_0)
\end{equation}
for every dyadic rectangle $R_{0} \subseteq [0,1]^{2}$.
A counterexample showing that \eqref{box1} does not imply \eqref{para2} was constructed by Carleson \cite{Car,Tao}.

It was observed by Chang \cite{Ch} (in a continuous setting) that \eqref{para2} is equivalent to the \emph{bi-parameter Carleson (or Carleson--Chang) condition}
\begin{equation}
\label{Carpara1}
\sum_{R\subset \Om} \beta_R^2\le C' m_2(\Om)\,,
\end{equation}
where the constant $C'$ is uniform for all subsets $\Omega \subseteq [0,1]^{2}$ that are finite unions of dyadic rectangles.
This necessary and sufficient condition was later used by Chang and Fefferman \cite{ChF1} to characterize the dual of the Hardy space on the bi-disc $H^1(\bD^2)$.
The same embedding holds in dimension $n>2$, from $L^2(m_n)$ to $\ell^2(T^n, \{\beta^2\})$.

\medskip

\subsection{Terminology and notation}
\label{sec:notation}
We begin with order-theoretic conventions.
\begin{definition}
A \emph{finite tree} $T$ is a finite partially ordered set such that, for every $\omega\in T$, the set $\{\alpha\in T \colon \alpha\geq\omega\}$ is totally ordered (we allow trees to have several maximal elements).

An \emph{$n$-tree} $T^{n}$ is a cartesian product of $n$ (possibly different) finite trees with the product order.

A subset $\mathcal{U}$ (resp.\ $\mathcal{D}$) of a partially ordered set $T$ is called an \emph{up-set} (resp.\ \emph{down-set}) if, for every $\alpha\in\mathcal{U}$ and $\beta\in T$ with $\alpha \leq \beta$ (resp.\ $\beta \leq \alpha$), we also have $\beta\in\mathcal{U}$  (resp.\ $\beta\in\mathcal{D}$).
\end{definition}

The \emph{Hardy operator} on an $n$-tree $T^{n}$ is defined by
\begin{equation}
\label{eq:Hardy}
\bfI \phi (\gamma) := \sum_{\gamma'\ge \gamma} \phi(\gamma')
\quad\text{for any } \phi: T^{n}\to \bR.
\end{equation}
In the one-parameter case $n=1$ we denote it by $I$, and in the two-parameter case $n=2$ by $\bbI$.
The adjoint $\bfI^*$ of the Hardy operator $\bfI$ is given by the formula
\begin{equation}
\label{eq:adj-Hardy}
\bfI^* \psi (\gamma)= \sum_{\gamma'\le \gamma} \psi(\gamma').
\end{equation}
\begin{definition}
Let $\mu,w$ be positive functions on $T^{n}$.
The \emph{box constant} is the smallest number $[w,\mu]_{Box}$ such that
\begin{equation}
\label{eq:box}
\mathcal{E}_{\beta}[\mu] :=
\sum_{\alpha \leq \beta} w(\alpha) (\bfI^{*}\mu(\alpha))^{2}
\leq [w,\mu]_{Box}
\sum_{\alpha \leq \beta} \mu(\alpha),
\quad \forall \beta\in T^{n}.
\end{equation}
The \emph{Carleson constant} is the smallest number $[w,\mu]_{C}$ such that
\begin{equation}
\label{eq:Carleson}
\sum_{\alpha\in\mathcal{D}} w(\alpha) (\bfI^{*}\mu(\alpha))^2 \leq [w,\mu]_{C} \mu(\mathcal{D}),
\quad \forall \mathcal{D}\subset T^{n} \text{ down-set.}
\end{equation}
The \emph{hereditary Carleson constant (or restricted energy condition constant or REC constant)} is the smallest constant $[w,\mu]_{HC}$ such that
\begin{equation}
\label{eq:hereditary-Carleson}
\mathcal{E}[\mu \one_{E}]
=
\sum_{\alpha\in T^{2}} w(\alpha) (\bfI^{*} (\mu\one_{E})(\alpha))^2
\le [w,\mu]_{HC} \mu(E),
\quad \forall E \subset T^{n}.
\end{equation}
The \emph{Carleson embedding constant} is the smallest constant $[w,\mu]_{CE}$ such that the adjoint embedding
\begin{equation}
\label{bIstar}
\sum_{\alpha\in T^{2}} w(\alpha) \abs{\bfI^*( \psi\mu)(\alpha) }^2
\leq [w,\mu]_{CE} \sum_{\omega \in T^{2}} \abs{\psi(\omega)}^2 \mu (\omega)
\end{equation}
holds for all functions $\psi$ on $T^{n}$.
\end{definition}

For positive numbers $A,B$, we write $A \lesssim B$ if $A\leq CB$ with an absolute constant $C$, that in particular does not depend on the tree or bi-tree or the weights $w,\mu$.

\subsection{Main result}
The inequalities
\begin{equation}
\label{eq:obvious}
[w,\mu]_{Box} \le [w,\mu]_{C} \leq [w,\mu]_{HC} \leq [w,\mu]_{CE}
\end{equation}
are obvious.
The converse inequalities for $1$-trees were proved in \cite{NTV99}.
For $2$-trees, in the case $w\equiv 1$, the converse inequality
\[
[1,\mu]_{CE} \lesssim [1,\mu]_{C}
\]
was proved in \cite{AMPS}.
In \cite{AMPVZ}, it was proved that, more generally,
\[
[w,\mu]_{CE} \lesssim [w,\mu]_{Box}
\]
for weights $w$ of tensor product form on $2$-trees.
In this article, we extend this result to $3$-trees.
\begin{theorem}
\label{thm:main}
Let $\mu: T^3\to [0, \infty)$. Let $w: T^3\to [0, \infty)$ be of tensor product form.
Then the reverses of the inequalities in \eqref{eq:obvious} also hold:
\[
[w,\mu]_{CE}
\lesssim [w,\mu]_{HC}\lesssim [w,\mu]_{C} \lesssim [w,\mu]_{Box}.
\]
\end{theorem}
Theorem~\ref{thm:main} will follow from conditional results on $n$-trees, namely Theorem~\ref{thm:HC=>CE} and Theorem~\ref{thm:box=>HC}.

\section{Holomorphic function spaces in polydisc}
\label{compl}

Another way to interpret the Hardy inequality (or more precisely, its weighted version, see below) is to consider its connection to certain problems in the theory of Hilbert spaces of analytic functions on the (poly-)disc.It was actually this connections that motivated the study of this inequality in \cite{ARSW} and \cite{AMPS}.

We start with some additional notation. Given an integer $d\geq1$ and $s = (s_1,\dots,s_d) \in\mathbb{R}^d$ we consider a Hilbert space $\mathcal{H}_s(\mathbb{D}^d)$ of analytic functions on the poly-disc $\mathbb{D}^d$ with the norm
\[
\|f\|_{\mathcal{H}_s(\mathbb{D}^d)}^2 := \sum_{n_1,\dots,n_d\geq0} |\widehat{f}(n_1,\dots,n_d)|^2(n_1+1)^{s_1}\cdot\dots\cdot(n_d+1)^{s_d},
\]
where
\[
f(z) = \sum_{n_1,\dots,n_d\geq0}\widehat{f}(n_1,\dots,n_d)z_1^{n_1}\cdot\dots\cdot z_d^{n_d},\quad z= (z_1,\dots,z_d) \in\mathbb{D}^d.
\]
Observe, that, clearly
\begin{equation}\label{e:c71}
\mathcal{H}_{\vec{s}}(\mathbb{D}^d) = \bigotimes_{j=1}^d\mathcal{H}_{s_j}(\mathbb{D}).
\end{equation}
In particular, the choice $s = (0,\dots,0)$ gives a classical Hardy space on the poly-disc, on the other hand $s = (1,\dots,1)$ corresponds to the Dirichlet space.

\subsection{Embedding (Carleson) measures on polydisk}

A measure $\nu$ on $\mathbb{D}^d$ is called a Carleson measure for $\mathcal{H}_s$, if there exists a constant $C_{\nu}$ such that
\begin{equation}\label{e:c70}
\int_{\mathbb{D}^d}|f(z)|^2 \,d\nu(z) \leq C_{\nu}\|f\|^2_{\mathcal{H}_s(\mathbb{D}^d)},
\end{equation}
or, in other words, the embedding $Id:\,\mathcal{H}_s(\mathbb{D}^d)\rightarrow L^2(\mathbb{D}^d,d\nu)$ is bounded.

For brevity we concentrate below on the case $d=2$, indicating the changes necessary for other $d$. Consider first the case of $s=0$.

Given a holomorphic function $f(z_1,z_2)=\sum_{m,n\geq0}a_{mn}z_1^mz_2^n$ on $\mathbb{D}^2$ we let
\begin{equation}\notag
\|f\|^2_{\mathcal{D}(\mathbb{D}^2)} = \sum_{m,n\geq 0}|a_{mn}|^2(m+1)(n+1),
\end{equation}
this norm can also be written as follows
\begin{equation}\notag
\begin{split}
&\|f\|^2_{\mathcal{D}(\mathbb{D}^2)} = \frac{1}{\pi^2}\int_{\mathbb{D}^2}|\partial_{z_1,z_2}f(z_1,z_2)|^2\,dz_1\,dz_2+\frac{1}{2\pi^2}\int_{\mathbb{T}}\int_{\mathbb{D}}|\partial_{z_1}f(z_1,e^{it})|^2\,dz_1\,dt +\\
&\frac{1}{2\pi^2}\int_{\mathbb{D}}\int_{\mathbb{T}}|\partial_{z_2}f(e^{it},e^{is})|^2\,ds\,dz_2 + \frac{1}{4\pi^2}\int_{\mathbb{T}}\int_{\mathbb{T}}|f(e^{is},e^{it})|^2\,ds\,dt = \\
&\|f\|^2_* + \;\textup{other terms},
\end{split}
\end{equation}
where $\|f\|_*$ is a seminorm which is invariant under biholomorphisms of the bidisc. In what follows however we use an equivalent norm, arising from the representation $\mathcal{D}(\mathbb{D}^2) = \mathcal{D}(\mathbb{D})\otimes\mathcal{D}(\mathbb{D})$ (this particular choice will be justified in few lines). For $f\in Hol(\mathbb{D})$ let
\begin{equation}\label{e:21}
\|f\|^2_{\mathcal{D}} := \frac{1}{\pi}\int_{\mathbb{D}}|f'|^2(z)\,dz + C_0|f(0)|^2,
\end{equation}
where $C_0>0$ is a constant to be chosen shortly.
It is classical fact that the Dirichlet space on the unit disc is a Reproducing Kernel Hilbert Space (cite [literature]), and, consequently,  $\mathcal{D}(\mathbb{D}^2)$ is one as well. The reproducing kernel $K_z,\; z\in\mathbb{D}^2$ (generated by $\|\cdot\|_{\mathcal{D}}$) is
\begin{equation}\label{e:22}
K_z(w) = \left(C_1 + \log\frac{1}{1-\bar{z}_1w_1}\right)\left(C_1 + \log\frac{1}{1-\bar{z}_2w_2}\right),\quad z,w\in\mathbb{D}^2
\end{equation}
(so it is a product of reproducing kernels for $\mathcal{D}(\mathbb{D})$ in respective variables), and $C_1 > 0$ is a constant depending on $C_0$.\\
The definition of norm in \eqref{e:21} implies that $K_z$ enjoys the following important property
\begin{equation}\label{e:23}
\Re K_z(w) \sim |K_z(w)|,\quad z,w\in\mathbb{D}^2,
\end{equation}
if we take $C_1$ (re. $C_0$) to be large enough


Let $\mu,w : T^d\rightarrow \mathbb{R}_+$. We define a weighted Hardy operator to be
\[
\mathbf{I}_wf(\alpha) := \sum_{\beta\geq\alpha}f(\beta)w(\beta).
\]
We call $(\mu,w)$ a trace pair for the weighted Hardy inequality, if
\begin{equation}\label{e:c73.1}
\int_{T^d}(\mathbf{I}_wf)^2d\mu \lesssim \int_{T^d}f^2\,dw
\end{equation}
for any $f:T^2\rightarrow \mathbb{R}_+$, i.e the operator $\mathbf{I}_w:\, L^2(T^d,dw)\rightarrow L^2(T^d,d\mu)$ is bounded. The dual version is
\begin{equation}\label{e:c73.2}
\int_{T^d}(\mathbf{I}^*(\varphi\mu))^2\,dw \leq \int_{T^d}\varphi^2\,d\mu
\end{equation}
for any $\varphi: T^d\rightarrow\mathbb{R}_+$, where 
\[
\mathbf{I}^*\varphi(\beta) := \sum_{\alpha\leq\beta}\varphi(\alpha).
\]

It turns out that trace pairs for the weighted Hardy inequality and Carleson measures for $\mathcal{H}_s$ are closely related. Below we give a brief overview of this relationship. We gloss over most of the technical parts of this short exposition, for more details see \cite{ARSW} and \cite[Section 2]{AMPS}, where it was presented for $d=1, s=s_1 \in (0,1]$ and $d=2, s = 1$ respectively.

We start by assuming that $s\in (0,1]^d$ (so that $\mathcal{H}_s(\mathbb{D}^d)$ is a \textit{weighted Dirichlet space on the poly-disc}), and that $\supp\nu\subset r\mathbb{D}^d$ for some $r<1$ (the latter is just a convenience assumption that allows us to make the corresponding graphs to be finite, no estimate below will depend on $r$, or on the depth of the graph).

It is well known that $\mathcal{H}_{s_j}(\mathbb{D}),\,1\leq j\leq d,$ is a \textit{reproducing kernel Hilbert space} (RKHS) with kernel $K_{s_j}$ satisfying (possibly after a suitable change of norm)
\begin{equation}\label{e:c74}
\begin{split}
& |K_{s_j}|(z_j,\zeta_j) \asymp |1-z_j\bar{\zeta}_j|^{s_j-1},\quad 0<s_j<1\\
& |K_{s_j}|(z_j,\zeta_j) \asymp \log|1-z_j\bar{\zeta}_j|^{-1},\quad s_j=1.
\end{split}
\end{equation}
Moreover it is not hard to verify that 
\begin{equation}
\label{Re}
\Re K_{s} \asymp |K_{s}|,  0< s\le 1\,.
\end{equation}
However, the case $s=0$ is a special case as
\begin{equation}
\label{Poi}
\text{Poisson kernel is not equivalent to the absolute value of Cauchy kernel}\,.
\end{equation}
It follows immediately that $\mathcal{H}_{\vec{s}}(\mathbb{D}^d)$ is a reproducing kernel Hilbert space as well, and
\[
K_{\vec{s}}(z,\zeta) = \prod_{j=1}^dK_{s_j}(z_j,\zeta_j),\quad z,\zeta\in\mathbb{D}^d.
\]
Going back to the Carleson embedding we see that $Id: \mathcal{H}_{\vec{s}}(\mathbb{D}^d)\rightarrow L^2(\mathbb{D}^d,d\nu)$ is bounded if and only if its adjoint $\Theta$ is bounded as well. Let us compute its action on a function $g\in L^2(\mathbb{D}^d,d\nu)$
\begin{equation}\notag
(\Theta g)(z) = \langle\Theta g, K_{\vec{s}}(z,\cdot)\rangle_{\mathcal{H}_{\vec{s}}(\mathbb{D})} = \langle g,K_{\vec{s}}(z,\cdot)\rangle_{L^2()\mathbb{D}^d,d\nu} = \int_{\mathbb{D}^d}g(\zeta)\overline{K_{\vec{s}}(z,\zeta)}\,d\nu(\zeta).
\end{equation}
Hence, for $\Theta$ to be bounded it must satisfy
\begin{equation}
\label{e:c75}
\begin{split}
\|g\|^2_{L^2(\mathbb{D}^d,d\nu)} \gtrsim \|\Theta g\|_{\mathcal{H}_{\vec{s}}(\mathbb{D}^d)} = \langle g,\Theta g\rangle_{L^2(\mathbb{D}^d,d\nu)} = \int_{\mathbb{D}^{2d}}g(z)\overline{g(\zeta)}K_{\vec{s}}(z,\zeta)\,d\nu(z)\,d\nu(\zeta).
\end{split}
\end{equation}

If inequality \eqref{e:c75} holds then trivially the following holds:
\begin{equation}
\label{positive}
\begin{split}
\|g\|^2_{L^2(\mathbb{D}^d,d\nu)} \gtrsim \int_{\mathbb{D}^{2d}}g(z)g(\zeta)K_{\vec{s}}(z,\zeta)\,d\nu(z)\,d\nu(\zeta), \quad g\ge 0\,.
\end{split}
\end{equation}

If we would know that the real part of the coordinate reproducing kernel is comparable to its absolute value, we deduce that $\Theta$ is bounded, if and only if
\begin{equation}
\label{e:c75.1}
 \int_{\mathbb{D}^{2d}}g(z)g(\zeta)|K_{\vec{s}}(z,\zeta)|\,d\nu(z)\,d\nu(\zeta)\lesssim \|g\|^2_{L^2(\mathbb{D}^d,d\nu)}\end{equation}
for any positive $g$ on $\mathbb{D}^d$.

In fact, \eqref{e:c75} implies \eqref{positive}, and we can take the real part of both sides of \eqref{positive}, putting real part on kernel. Now if to know that
\begin{equation}
\label{ReVec}
\Re K_{\vec{s}}(z, \zeta)=\Re\prod_{j=1}^dK_{s_j}(z_j,\zeta_j)\asymp |\prod_{j=1}^dK_{s_j}(z_j,\zeta_j)|=| K_{\vec{s}}(z, \zeta)|,\quad z,\zeta\in\mathbb{D}^d,
\end{equation}
we would deduce   \eqref{e:c75}$\Rightarrow$\eqref{e:c75.1}. The only thing we need for this implication is the above pointwise equivalence \eqref{ReVec}.
On the other hand, the implication  \eqref{e:c75.1}$\Rightarrow$\eqref{e:c75} obviously always holds.

\medskip

We conclude that in the presence of pointwise equivalence \eqref{ReVec} we have \eqref{e:c75}$\equiv$\eqref{e:c75.1}.

\medskip

However, equivalence \eqref{ReVec}--ultimately important for us to prove equivalence of dyadic and analytic embeddings (see below)--has limitations. First of all it is false even for $1D$ case $d=1$ if $s=0$, see \eqref{Poi}. That makes the case $s=0$ quite special. It is well known that for $1D$ case embedding measures for Poisson and Cauchy kernels on $L^2(\bT)$ are the same. This is rather simple, but should be consider as ``a miracle". Already in $2D$ situation the fact that
embedding measures for Poisson $P_{z_1}P_{z_2}$ and Cauchy $K_{\vec{0}}(z, \zeta)=(1-z_1\bar\zeta_1)^{-1} (1-z_2\bar\zeta_2)^{-1} $ kernels on $L^2(\bT^2)$ are the same is a subtle fact that will be considered in  \cite{MTV} separately. It is based on Ferguson--Lacey's characterization of symbols of ``little" Hankel operators \cite{FL}, \cite{L1}.

Another interesting distinction of the case $s=0$ is again about \eqref{Poi}. The reader will see, that for $s>0$ we will characterize the embedding in terms of
simple box (rectangular) test. As it is well known from the works of Chang, Fefferman and Carleson \cite{Ch}, \cite{RF}, \cite{Car}, \cite{Tao}, such characterization is not possible for Poisson embedding of $L^2(\bT^d)$ if $d\ge 2$. We would wish to attribute this phenomena to the fact that Poisson kernel has a special shape. In our language this means that unlike \eqref{dK}  below that holds  for $s\neq 0$, the inequality
\begin{equation}
\label{pK}
\mathbf{I}_{w_{\vec{0}}}\mathbf{1}(\alpha\vee\beta) \lesssim P(\alpha,\beta),
\end{equation}
is often false, where $P(z, \zeta)$ is the (multi-parameter) Poisson kernel, and
 $$
 P(\alpha, \beta):=\sup_{z\in q(\al),\zeta\in q_{\beta}}P(z,\zeta)\,.
 $$
 
 This finishes the discussion of $\vec{s}=\vec{0}$. 
 
 Now let $\vec{s}=(s_j)$ and $0<s_j\le 1$. 
 
 \subsection{Unweighted Dirichlet space in polydisk}
 
 We first consider the case when all $s_j=1$. For brevity we assume $d=2$. For unweighted Dirichlet space this is not a restriction of generality as we will see soon.
The reproducing kernel $K_{\vec{1}}(z, \zeta) = \log (1-z_1\bar \zeta_1) \log (1-z_2\bar \zeta_2)= K_1(z_1, \zeta_1) K_1(z_2, \zeta_2)$. The first idea is to see that our inequality \eqref{e:c75} (equivalent to embedding):
\begin{equation}
\label{e:c75.2}
 \int_{\mathbb{D}^{2}}g(z)\overline{g(\zeta)}K_{\vec{1}}(z,\zeta)\,d\nu(z)\,d\nu(\zeta) \le A \|g\|^2_{L^2(\mathbb{D}^2,d\nu)}
 \end{equation}
 implies that for every $C\ge 0$ we have
\begin{equation}
\label{e:c75.3}
 \int_{\mathbb{D}^{2}}g(z)\overline{g(\zeta)}(C+K_{1}(z_1,\zeta_1))(C+ K_{1}(z_2,\zeta_2))\,d\nu(z)\,d\nu(\zeta) \le B(C) \|g\|^2_{L^2(\mathbb{D}^2,d\nu)}
 \end{equation} 
 To deduce the latter inequality from \eqref{e:c75.2} one  should open the brackets and consider $4$ terms in the LHS. The term with $K_{1}(z_1,\zeta_1)) K_{1}(z_2,\zeta_2)$ is $\lesssim \|g\|^2_{L^2(\mathbb{D}^2,d\nu)}$ by \eqref{e:c75.2}. The term with $C^2\int_{\mathbb{D}^{2}}g(z)\overline{g(\zeta)}d\nu(z)\,d\nu(\zeta)$ obviously is $\lesssim  \|g\|^2_{L^2(\mathbb{D}^2,d\nu)}$ by H\"older inequality. Consider one of mixed terms (they are treated symmetrically):
 \[
 C \int_{\mathbb{D}^{2}}g(z)\overline{g(\zeta)}K_{1}(z_1,\zeta_1)\,d\nu(z)\,d\nu(\zeta)=: CI,
 \]
skip $C$, and, using disintegration theorem and pushing forward of $\nu$ to the first coordinate (we call that push forward $\nu_1$), we write $I$ as follows
\[
I= \int_{\mathbb{D}}G(z_1)\overline{G(\zeta_1)}K_{1}(z_1,\zeta_1))\,d\nu_1(z_1)\,d\nu_1(\zeta_1),
\]
where $G(w):= \int g(w, u) d\nu_w(u)$ and $d\nu_w(u)$ are slicing measures: $\nu(E)=\int \nu_w(E) d\nu_1(w)$.

Push forward measure $\nu_1$ on $\bD$ is  obviously a Carleson measure for $1D$ Dirichlet space, if $\nu$ is a Carleson measure for Dirichlet space in $2D$.
Therefore,
\begin{equation*}
\begin{split}
&\int_{\mathbb{D}}G(z_1)\overline{G(\zeta_1)}K_{1}(z_1,\zeta_1))\,d\nu_1(z_1)\,d\nu_1(\zeta_1) \le B\int_{\bD} |G_1(z_1)|^2 \, d \nu_1(z_1)\le 
\\
&B\int_{\bD} \Big(\int_{\bD} |g(z_1, z_2)| \, d\nu_{z_1}(z_2)\Big)^2\, d\nu_1(z_1) \le B' \int_{\bD^2}|g(z_1, z_2)|^2  \, d\nu_{z_1}(z_2)d\nu_1(z_1)  \le
\\
&B' \int_{\bD^2}|g(z_1, z_2)|^2  \, d\nu(z)\,.
\end{split}
\end{equation*}

We deduced \eqref{e:c75.3}  from \eqref{e:c75.2} by the use of the disintegration theorem and slicing measures. Notice that the nature of the kernel did not play any role. We could have done this with any dimension $d$ and any kernel $K_{\vec{s}}$ instead of $K_{\vec{1}}$.

But now the fact that we worked with precisely $K_{\vec{1}}$ will be crucial. In fact, values of $ K_1$ are obviously in the right half-plane (argument of logarithm of $1-z\bar\zeta, z, \zeta \in \bD$), hence
\begin{equation}
\label{Im}
|\Im K_1(z, \zeta)| \le \pi.
\end{equation}
Hence by adding  sufficiently large constant $C>0$ to $K_1(z, \zeta)$ we achieve a) $|\Re (C+K_1)|>>|\Im (C+K_1)|$, b) $|\Re (C+K_{\vec{1}}(z, \zeta))|\ge c\Re (\Pi_{j=1}^d(C+K_{1})(z_j, \zeta_j)))$ for any dimension $d$, it is enough to choose $C=C(d)$ large positive number. The latter inequality implies that
\begin{equation}
\label{all1}
\Re \Pi_{j=1}^d(C+ K_{1}(z_j, \zeta_j)) \asymp | \Pi_{j=1}^d(C+ K_{1}(z_j, \zeta_j))|\,.
\end{equation}

Therefore, for $\vec{s}=\vec{1}$ by modifying  the kernel we can achieve \eqref{ReVec} without changing the class of Carleson measures as \eqref{e:c75.3} shows.
This means that without changing the set of embedding measures we can equivalently replace inequality \eqref{e:c75} by \eqref{e:c75.1}. This is only for $\vec{s}=\vec{1}$ (but for any dimension $d$).

\bigskip

 \subsection{Weighted Dirichlet space in polydisk}
 
 Now $\vec{s}=(s_j)_{j=1}^d, 0<s_j \le 1$, but $\vec{s}\neq \vec{1}$. We are unable to repeat the trick that was successful in the previous section. In fact, for $K_s=(1-z\bar\zeta)^{s-1}$ with $0<s<1$ \eqref{Im} does not hold, the imaginary part will not be bounded, and so the previous reasoning with adding a large constant to each kernel of each variable does not work.
 
 However, to reduce the analytic embedding \eqref{e:c75} to dyadic embedding on multi-tree we seem to really need to show that \eqref{e:c75} implies \eqref{e:c75.1} (the converse implication being always trivial).
 
 Here we have only partial results, namely for the case when 
 \begin{equation}
 \label{close1}
 1-\epsilon(d) \le s_j\le 1
 \end{equation}
 for $\epsilon(d)$ sufficiently close to $0$.
 
 We just notice that $1-z\bar \zeta$ lies in the right half-plane if $z, \zeta\in \bD$, and so $(1-z\bar\zeta)^{\epsilon}$ lies in the cone $C_\epsilon=\{u+iv, u\ge 0, |v| \le u \cdot\tan \pi\epsilon\} $. Therefore, for every $s_j\in (1-\epsilon, 1)$,
 \[
 |\Im K_{s_j} (z_j, \zeta_j)| \le \tan \pi\epsilon\cdot \Re K_{s_j} (z_j, \zeta_j)\,.
 \]
 This implies that if $\epsilon$ is sufficiently small (depending on the dimension $d$) then
 \eqref{ReVec} holds, which, as we have already explained gives us the equivalence of \eqref{e:c75} and \eqref{e:c75.1}.
 
 From \eqref{e:c75.1} we will now proceed to conclude that dyadic embedding holds. Then we will explain why dyadic embedding implies \eqref{e:c75.1}, thus closing the circular argument.
 
 \bigskip

\subsection{From embedding of analytic functions in the polydisc to dyadic multi-parameter embedding}
Consider a fixed dyadic lattice $\cD$ on $\mathbb{T}$. Consider now the classical Whitney decomposition of $\mathbb{D}$ into dyadic Carleson half-boxes. It corresponds to this dyadic lattice. Clearly there is a one-to-one correspondence between these boxes and the vertices of a dyadic tree $T$ just because vertices of $T$ and dyadic intervals of $\cD$ are in one-to-one correspondence. So each box has an address $\al$, which is a vertex of $T$.  We can choose a fixed dyadic lattice for each coordinate tori $\mathbb{T}$. Consequently the Whitney decomposition of $\mathbb{D}^d$ generated by Cartesian products of the respective coordinate decompositions can be encoded by vertices of $T^d$, i.e. each (multi-)box $q$ corresponds to a point $\alpha_q\in T^d$, and vice-versa, each $\alpha\in T^d$ has a unique counterpart $q(\al)$.

 The reader should keep it in mind when we
will consider boxes constructed by random choice of dyadic lattices $\om:=(\cD_1,\dots, \cD_d)$. Notice that the collections $\om:=(\cD_1,\dots, \cD_d)$ of dyadic lattices form a natural measure space provided with probability measure: $(\Om, \bP)$.  For future purposes notice that given a point $z$ in polydisc $\bD^d$, and a random multi-lattice $\om$, we will call the address of the  box that contains $z$ by symbol $\al^\om(z)$ (any fixed $z$ is contained in an open box almost surely, and, thus, the address is uniquely defined by $z$ and $\om$). The box should be called $q(\al^\om(z))$.  Often we skip $\om$.

As a result we can define a family canonical map $\Lambda=\Lambda^\om : Meas^+(\mathbb{D}^d) \rightarrow Meas^+(T^d)$ given by
\begin{equation}
\label{nuTOtree}
\Lambda\nu(\alpha) = \nu(q(\al)).
\end{equation}
Similarly, given a function $g\in L^2(\mathbb{D}^d,d\nu)$ we write
\[
\Lambda g(\alpha) := \frac{1}{\nu(q(\al))}\int_{q(\al)}g(z)\,d\nu(z).
\]


Define a random kernel as follows. Fix $\om \in \Om$ and $(z, \zeta)\in \bD^{2d}$, in dyadic multi-lattice $\om$ find $\al^\om, \beta^\om$ such that $z\in q(\al^\om), \zeta \in q(\beta^\om)$. Up to measure zero of $\om$, $z, \zeta$ lie in corresponding open boxes, hence, the boxes are uniquely defined, and so $\al^\om, \beta^\om$ are well-defined. Then consider 
$$
k^\om (z, \zeta):=(\mathbf{I}_{w_{\vec s}} \1)\!(\al^\om(z)\vee\beta^\om(\zeta))\,.
$$
where $\alpha\vee \beta$ is the least common ancestor of $\alpha$ and $\beta$ in geometry of $T^d$.  In particular, for $\vec{s}=\vec{1}$,  multi-tree kernel
$\mathbf{I}_{w_{\vec{1}}}\mathbf{1}(\alpha\vee\beta)$ is the number of ancestors that are common for $\alpha$ and $\beta$. If $\vec{s}\neq \vec{1}$, the kernel counts the weighted number of ancestors.

An elementary computation gives  that independently of $\om$ the following inequality holds if $s_i\neq 0, i=1, \dots, d$:
\begin{equation}
\label{dK}
k_{\vec{s}}^\om (z, \zeta) \lesssim |K_{\vec s}|(z,\zeta),
\end{equation}
The implied constant depends only on $d$ and $s_i\neq 0, i=1, \dots, d$.

\begin{remark}
\label{s000}
If all $s_i$ vanish, we have ``a phase transition" in the kernel, and \eqref{dK} stops to be true in general. This explains the special role of Hardy spaces on the polydisc. If the reader thinks that Chang--Fefferman theory gives the embedding theorem for Hardy space $H^2(\bD^d)$ (the case $s_i= 0, i=1, \dots, d$), we should upset the reader by saying that this is not so. Chang--Fefferman theory gives the characterization of embedding measures in $d$-harmonic space $h^2(\bD^d)$. As, obviously, the Hardy space of holomorphic functions in the polydisc is such that $H^2(\bD^d) \subset h^2(\bD^d)$, the Chang--Fefferman theory gives the sufficient condition for measure to be an embedding measure for the Hardy class, but whether it is a necessary condition (we believe it is) is not known outside the classical case $d=1$. If the influential paper \cite{FL} were correct, then its proof can be modified to give this necessity, but unfortunately the note \cite{V} indicated a counterexample to the reasoning (but not to the result) of \cite{FL}.
\end{remark}

The inverse inequality is generally not true pointwise (due to the difference between hyperbolic geometry on the unit disc and that of a dyadic tree.
However,  one can verify that if one considers the family of dyadic lattices $\Omega$  $ \om=(\cD_1, \dots, \cD_d)$ on $\mathbb{T}^d$ with a natural probability measure on this family, then with a fixed probability
\begin{equation}
\label{dKprob}
\begin{split}
&\forall (z, \zeta) \in \bD^{2d} \,\,\, \exists\,  \Omega(z, \zeta)\subset \Om: \, a) \, \,\, \mathbb{P} (\Omega(z, \zeta)) \ge c_d>0,\, b) \,\forall   \om \in \Omega(z, \zeta),
\\
&|K_{\vec{s}}|(z,\zeta) \le C_d k_{\vec{s}}^\om(z, \zeta)\,.
\end{split}
 \end{equation}
where $c_d, C_d$ are constant that depends only on dimension $d$. 

Now \eqref{dK}, \eqref{dKprob} give us
\begin{equation}
\label{AVK}
K_{\vec{s}}(z, \zeta)  \asymp 
\bE k_{\vec s}^\om  \mathbf{1}(z, \zeta)\,.
\end{equation}
In its turn \eqref{AVK} follows from \eqref{dKprob} and \eqref{dK}.
By Tonelli's theorem we have
\begin{equation}
\label{tonelli}
\sum_{\alpha\in T^d}\sum_{\beta\in T^d}\Lambda g(\alpha)\Lambda g(\beta)\mathbf{I}_{w_{\vec{s}}}\mathbf{1}(\alpha\vee\beta)\Lambda\nu(\alpha)\Lambda\nu(\beta) = \int_{T^d}(\mathbf{I}^*(\Lambda g\Lambda\nu))^2\,dw_{\vec{s}},
\end{equation}

Given that we fix $\om$ and write for $z\in \bD^d, \zeta\in bD^d$, we use that $k{\vec{s}}^\om(z, \zeta)$ is constant on each pair
of boxes from multi-lattice $\om$ detected by pair $(z, \zeta)$:
\begin{equation}
\label{AV}
\begin{split}
& \int_{\bD^d}\int_{\bD^d} g(z)g(\zeta) k_{\vec s}^\om d\nu(z) d\nu(\zeta) =\\
&\sum_{q(\al^\om)}\sum_{q(\beta^\om)} \Lambda g (q(\al^\om)) \Lambda g (q(\al^\om)) \mathbf{I}_{w_{\vec s}}\1(\al^\om\vee \beta^\om) \Lambda (q(\al^\om))\Lambda( (q(\beta^\om))\lesssim\\
& \|\Lambda g\|_{L^2(\Lambda\nu)}^2 \le \|g\|_{L^2(\nu)}^2\,.
\end{split}
\end{equation}
where constants of equivalence depend only on the dimension. Here we used \eqref{tonelli} and the boundedness  of operator with kernel  $\mathbf{I}_{w_{\vec{s}}}\mathbf{1}(\alpha\vee\beta)$ on graph $T^d$. 

Now let us hit \eqref{AV} by expectation in $\om$ and use \eqref{AVK}.
Therefore \eqref{e:c75.1} follows from \eqref{e:c73.2} for $\mu = \Lambda \nu$ and $w = w_{\vec{s}}$.

\par

Assume now that \eqref{e:c75.1} holds.  Fix a measure $\mu$  on $T^d$. Fix any $\om$. Let $\nu$ be any measure on $\bD^d$ such that $\Lambda\nu=\mu$. Then 
\[
 \int_{\bD^d}\int_{\bD^d} g(z)g(\zeta) k_{\vec s}^\om d\nu(z) d\nu(\zeta) \lesssim  \|g\|_{L^2(\nu)}^2
 \]
just because of \eqref{dK}. Apply this inequality to special non-negative $g$ that assume constant values on each given box $q(\al^\om)$. 
We can choose those constants arbitrarily with only condition that $\|g\|_{L^2(\nu)}^2= \|g\|_{L^2(\Lambda\nu)}^2<\infty$. Then we get \eqref{e:c73.2} for $\mu = \Lambda \nu$ and $w = w_{\vec{s}}$.

\subsection{Verifying \eqref{dKprob}}
It is enough to verify it for $d=1$ because then we can use  the product structure of the kernel $|K_{\vec{s}}|$ and the independence of lattices $\cD_1, \dots, \cD_d$.
Put $D(z, \zeta) := |u-v|+ 1-|z| + 1-|\zeta|$,  it is a sort of distance. Then 

$$
K_{s}(z, \zeta)=|1-z\bar{\zeta}|^{s-1}\approx D(z, \zeta)^{s-1}\,.
$$
Now  we define the analogous dyadic distance that depends on a dyadic lattice, call lattice $L$.
$D^L(z, \zeta)$ is defined as  the smallest length of dyadic arc from $L $ that is larger than 
$\max (1-|z|, 1-|\zeta|)$  and contains the shorter arc that has end-points $u, v$.

Then right hand side of  \eqref{dKprob} (for $d=1$) is $\approx (D^L(z, \zeta))^{s-1}$ and $D^L$ is always $\ge$  $c\,D$, where $c$ is an absolute constant. Of course we have
$$
\Big ( \text{dist}^L(z,\zeta)\Big)^{s-1}\approx\textbf{I}_{w_{s}}\one(\al\vee\beta)\,.
$$
To prove \eqref{dKprob} (for $d=1$) it is enough to prove that 
$$
\text{dist}^L(z,\zeta)\le C\, D(z, \zeta)
$$ for a set of dyadic lattices of a fixed probability. Let the full family of dyadic lattices be just the rotation of one fixed lattice provided with a natural probability measure $d\theta/2\pi$.

Let $I$ be a dyadic arc of length $2\pi\cdot 2^{-m}$. Given $z, \zeta$, let us calculate the probability of being a bad dyadic lattice, where bad means  the display inequality above is false with constant $C=8$. Each dyadic lattice has two end-points of first division, four end-points of the second division, et cetera.

Then the probability for the first division points to be inside $I$ is  $\frac{2}{\frac{2\pi}{|I|}}=2\cdot 2^{-m}$ (as we have two such points). The probability for the second division points to be inside $I$ is  $4\cdot 2^{-m}$ . We continue until we find the $m-4$-th division points for which the probability such a point is in $I$ is almost  $2^{m-4}\cdot 2^{-m} $. These are all bad scenarios. Their probability is at most  $1/8$.

Hence, the probability none of these points are in $I$ (which we can call a ``good'' event), is at least $7/8$. But if none of these division points are inside $I$, we have
$$
D_L (z, \zeta) \le 10 D(z, \zeta)\,.
$$
Inequality \eqref{dKprob} is proved.

\bigskip

\begin{remark}
\label{prWeights}
Notice that for $\vec{s}=\vec{1}$ (Dirichlet  space case) integration  in \eqref{tonelli} with  respect to $d w_{\vec{s}}$ means just summation over all vertices of $T^d$. For other $\vec{s}$ a natural weight appears (it weights the vertices), and the summation has to be with respect to this weight. In our situation of the scale $\mathcal{H}_{\vec{s}}$ (of various spaces of analytic functions in polydisc described at the beginning of this section),  the weight that appears is always the product of weights in each coordinate. {\it This emphasizes why we especially care about the results with product weights}.
\end{remark}

To summarize, the problem of characterizing Carleson measures for the weighted Dirichlet space $\mathcal{H}_{\vec{s}}$ can be often moved to a discrete medium (for $\vec{s}= \vec{1}$ can be {\it always} moved to discrete medium, for any dimension $d$), and  after that this problem interpreted (without any loss of information) as the problem of characterizing a trace pair $(\mu,w_{\vec{s}})$. For instance we see, that \eqref{e:c73.2} is equivalent to a single box condition (since $w_{\vec{s}}$ has a product structure)
\[
\sum_{\beta\le\alpha}(\mathbf{I}^*\Lambda\nu)^2(\beta)w_{\vec{s}}(\beta) \lesssim \mathbf{I}^*\Lambda\nu(\beta)
\]
for any $\beta\in T^d$. on the poly-disc this condition transforms to
\[
\sum_{R\subset Q}\nu^2(T(R))w_{\vec{s}}(R) \lesssim \nu(T(Q)),\quad \textup{for any}\; Q,
\]
where $Q,R$ are dyadic rectangles on the (poly-)torus $(\partial\mathbb{D})^d$, and $T(Q)$ is the usual tent area above $Q$. One can also check that this condition is necessary by testing  Carleson embedding on appropriate functions.

The argument above fails for a number of reasons, if even one of the parameters $s_j$ becomes zero. However, for the classical Hardy space on the polydisc one can still make a connection between Carleson embedding and Hardy inequality, only now we use the direct embedding \eqref{e:c73.1} instead of the dual \eqref{e:c73.2}, and the roles of $\mu$ and $w$ are reversed. It is done in Section \ref{s0}.

\section{End-point case $s=0$}
\label{s0}

We repeat ourselves: the equivalence \eqref{ReVec}--ultimately important for us to prove equivalence of dyadic and analytic embeddings--has limitations. First of all \eqref{ReVec} is false even for the case $d=1$ if $s=0$, see \eqref{Poi}. That makes the case $s=0$ quite special. It is well known that for $d=1$ case embedding measures for Poisson and Cauchy kernels on $L^2(\bT)$ are the same. This is rather classical, but should be consider as ``a miracle" exactly because \eqref{ReVec} fails. Already in $2D$ situation the fact that
embedding measures for Poisson $P_{z_1}P_{z_2}$ and Cauchy $K_{\vec{0}}(z, \zeta)=(1-z_1\bar\zeta_1)^{-1} (1-z_2\bar\zeta_2)^{-1} $ kernels on $L^2(\bT^2)$ are the same is a subtle fact that will be considered in  \cite{MTV} separately. It is based on Ferguson--Lacey's characterization of symbols of ``little" Hankel operators \cite{FL}, \cite{L1}.

Another interesting distinction of the case $s=0$ is  about \eqref{Poi}. The reader will see, that for $s>0$ we characterize the embedding in terms of
simple box (rectangular) test. As it is well known from the works of Chang, Fefferman and Carleson \cite{Ch}, \cite{RF}, \cite{Car}, \cite{Tao}, such characterization is not possible for Poisson embedding of $L^2(\bT^d)$ if $d\ge 2$. We would wish to attribute this phenomena to the fact that Poisson kernel has a special shape. In our language this means that unlike \eqref{dK}  above that holds  for $s\neq 0$,  the same type of inequality  for Poisson kernel
\begin{equation}
\label{pK1}
\mathbf{I}_{w_{\vec{0}}}\mathbf{1}(\alpha\vee\beta) \lesssim P(\alpha,\beta),
\end{equation}
is false, where $P$ is a multi-parameter Poisson kernel,
 $P(\alpha, \beta):=\sup_{z\in q(\al),\zeta\in q_{\beta}}P(z,\zeta) $.  

For $s=0$  the space $\calH_s(\mathbb{D}^d)=: H^2(\mathbb{D}^d)$ is the Hardy space on the polydisc. The embedding $Id: H^2(\mathbb{D}^d)\rightarrow L^2(\mathbb{D}^d,d\nu)$ can be still equivalently described as inequality \eqref{e:c75}, but cannot  be described any longer  as inequality \eqref{e:c76}. The reason is that the reproducing kernel $K_0(z, \zeta) = (1-z\bar\zeta)^{-1}$ does not satisfy anymore the property that its real part is equivalent to its absolute value.

Still we want to deduce the embedding theorem $Id: H^2(\mathbb{D}^d)\rightarrow L^2(\mathbb{D}^d, d\nu)$  from dyadic  statement of the type \eqref{e:c73.2}. 
Notice that embedding of Hardy space of analytic functions in the polydisc follows from the Poisson embedding. Also notice that for dimension $d=1$ these two embedding are equivalent, in the sense that the classes of embedding measures in the disc are the same.

This is absolutely not obvious for $d>1$. So below we consider only embedding of $L^2(\mathbb{T}^d)$ by the means of multi-Poisson kernel. We do not touch upon the question of equivalence  of this Poisson embedding of $L^2(\mathbb{T}^d)$ and the (Poisson) embedding of $H^2(\mathbb{T}^d)$. The relation between two embeddings (that of $L^2(\mathbb{T}^d)$ and that of $H^2(\mathbb{T}^d)$) for $d>1$ will be addressed in \cite{MTV}. It is a really subtle question that requires the extension of \cite{FL}.   To our utmost consternation this  question has not been addressed in the literature.  

To this end we stop to consider the adjoint operator to embedding $Id: H^2(\mathbb{D}^d)\rightarrow L^2(\mathbb{D}^d, d\nu)$. Instead we consider this embedding directly, namely, if $P^k$ denotes the Poisson extension in $k$-th variable, we write down our embedding as the following inequality
\begin{equation}
\label{Poi1}
\int_{\mathbb{D}^d}[P^1\dots P^d\,f]^2 \, d\nu \le \int_{\mathbb{T}^d} |f|^2 \, dm_d\,,
\end{equation}
where $\mathbb{T}^d$ is the torus and $m_d$ its Lebesgue measure.  We emphasize again that this should hold for any $f\in L^2(\bT^d, m_d)$. Let $\{q(\al)pha\}_{\alpha\in T^d}$ be the Whitney decomposition of $\mathbb{D}^d$ generated by Cartesian products of the respective coordinate decompositions. By \cite{Ch} we know that inequality \eqref{Poi1} is equivalent to Carleson--Chang condition:
\begin{equation}
\label{CCh}
\sum_{\al: q(\al)pha \cap Tent(\Om)\neq \emptyset} \nu(q(\al)pha) \le Cm_d(\Om)\quad \forall \,\, \text{open}\,\, \Om \subset {\mathbb T}^d\,.
\end{equation}

So we wish to deduce the implication  \eqref{CCh} $\Rightarrow$ \eqref{Poi1} by using only the dyadic multi-tree statement that we will formulate now.  

Let $f:T^d\to [0,\infty)$ and let $m_d$ be Lebesgue measure on $\pd T^d:= (\pd T)^d$ given by $m_d(\om)= 2^{-Nd}$. Now let $\Om$ be an arbitrary union of elementary cubes $\om$'s of size $2^{-N}$. Call such sets {\it dyadic open sets}.  For any $\al\in T^d$ we denote by $R_\al$ the dyadic $d$-subrectangle of the unit cube that corresponds to $\al$.
Let $\nu:T^d \to [0, \infty)$ be such that
\begin{equation}
\label{CChd}
\sum_{\al: R_\alpha \subset \Om}(m_d(R_\al))^2 \nu(\alpha) \le Cm_d(\Om)\quad \forall \,\, \text{dyadic open}\,\, \Om\subset \pd T^d \,.
\end{equation}
We consider the inequality on multi-tree $T^d$:
\begin{equation}
\label{emb-d}
\int_{T^d} \Big(\bfI^*(f dm_d)\Big)^2\, d\nu \le C_1 \int_{\pd T^d} f^2 \, dm_d\,.
\end{equation}

Suppose we know that \eqref{CChd} $\Rightarrow$ \eqref{emb-d} (with different constants, but without dependence on $N$). We want to use this implication as the only tool to prove implication \eqref{CCh} $\Rightarrow$ \eqref{Poi1}.

This requires some work even for the case $d=1$. Below is the way to do this reduction for $d=1, 2$. General $n$ follows the same steps.

For an interval $I$ of $\bR$, $Q_I$ denotes Carleson box, $T_I$ denotes its upper half. Similarly, for a rectangle $R=I\times J$ in $\bR^2$, we have
$Q_R:=Q_I\times Q_J$ and $T_R:= T_I\times T_J$. If $I$ run over a certain dyadic lattice of intervals, then $T_I$ tile the upper half-plane. Similarly, if $R$ run over dyadic system of rectangles, $T_R$ tile $\bC_+^2$. Let $I_0$ always denote $[-1,1]$, and let $Q_0$  be always  $Q_{[-1,1]}$.

\subsection{One dimensional case.}
\label{1}
 Let $Pf$ mean Poisson extension of $f$. We first consider  a 1D case. Let measure $\nu$ lie in the upper half plane, a nonnegative test function $f$ on the real line has support in $[-1/2, 1/2]$, and let measure $\nu$ satisfy the following box Carleson condition:
\begin{equation}
\label{1bC}
\nu(Q_I) \le C_1\,|I|\quad \forall I\,.
\end{equation}
We want to give a new proof of Carleson embedding:
\begin{equation}
\label{1PC}
\int_{Q_0} [Pf]^2 \, d\nu \le C_2 \int f^2 dx,
\end{equation}
where $C_2$ depends only on $C_1$.

As we have Harnack inequality for $Pf$ we always may assume that $\nu$ is a doubling measure in Poincar\'e metric of $\bC_+$.

We wish to prove implication \eqref{1bC} $\Rightarrow$ \eqref{1PC} by allowing ourselves to use only implication  \eqref{1bCd} $\Rightarrow$ \eqref{1Cd}, where given a dyadic lattice $\cD$, we have
\begin{equation}
\label{1bCd}
\nu(Q_I) \le C_1\,|I|\quad \forall I \in \cD
\end{equation}

\begin{equation}
\label{1Cd}
 \sum_{J\in \cD} \La f\Ra_J^2 \nu(T_J)\le C\int_\bR f^2 dx\,.
 \end{equation}

\bigskip

Here are several notations: as always for a given $I$, $\la I$ means the interval with the same center, but with lenght $\la |I|$.
If $I$ is an interval of a dyadic lattice $\cD$ then $I^j$ is its ancestor such that $|I^j|= 2^j |I|$. We denote by $c_I=x_I+iy_I$ the center of $T_I$, and by $P_I$ the Poisson kernel with pole at $c_I$. As $P_I f $ is bounded by an absolute constant times the convex combination of averages $\La f\Ra_{2^k \,I}$,  $k=0, \dots, \log\frac1{|I|}$ and average $\La f\Ra_{I_0}$, we can  choose $k_I$ that gives the maximum to $\La f\Ra_{2^k \,I}$,  $k=0, \dots, \log\frac1{|I|}$, and then
$$
Pf (c_I)=P_I f \le A_1 \La f\Ra_{2^{k_I} \,I} + A_2  \La f\Ra_{I_0}\,.
$$
Our goal is to give a new way to prove \eqref{1PC}. Traditionally it is deduced from \eqref{1bC} by interpolation argument. 
We wish to deduce it using dyadic $L^2$ estimate. The second term is trivial to estimate in \eqref{1PC}.

As to the first term, we will do the following. We consider the probabilistic space of dyadic lattices built as follows. 
Divide $\bR$ into equal intervals of size $2^{-N}$, where $N$ is very large. We do it to have $[-1/2, 1/2]$ tiled. Now we can toss the coin and choose which pair is united to one dyadic interval of size 
$2^{-N+1}$. These are fathers. Toss the coin again to choose who are grandfathers. Now for a given  interval of size $2^{-N}$ we have already $4$ different grandfathers, each with probability $1/4$.  We continue this tossing for total number of $N+4$ tossing. For any  interval of size $2^{-N}$ inside $[-1/2, 1/2]$ the most senior ancestor will contain $Q_0=[-1/2, 1/2]$ with probability $15/16$. We call the collection of such dyadic lattices $\Om$ (it is a finite family of lattices). All  dyadic lattices in $\Om$ have the equal probability, and we just renormalize the probability to have $\bP(\Om)=1$.

\medskip

The thus obtained random dyadic lattice will be called $\cD(\om)$, their probability space will be called  $(\Om, \bP)$. 
Now fix $\om \in \Om$ (meaning fix one of those lattices), and consider some small $I\in \cD(\om)$ of size $2^{-N}$. We consider $c_I$ and find $k_I$ as above.  Consider $2^{k_I}I$. It is not dyadic may be, but it has the same center $x_I$ as dyadic $I$, so consider $I^{k_I}$ and $I^{k_I+10}$ and  check whether $I^{k_I}$ is inside $\frac34 I^{k_I+10}$. Suppose yes. Then obviously as $x_I\in I^{k_I}$, we will have that 
$$
2^{k_I}I\subset I^{k_I+10}\,.
$$
It is very easy to see that 
$$
\bP \Big\{ I^{k_I}\,\text{ is inside}\,\frac34 I^{k_I+10}\Big\}\ge 1/2\,.
$$
Thus 
$$
\bP \Big\{ 2^{k_I}I\subset I^{k_I+10}\Big\}\ge 1/2\,.
$$
If the event $2^{k_I}I\subset I^{k_I+10}$ happened, then we call $T_I$  good, we color it red, we color $I^{k_I+10}$ also red, but we take measure $\nu$ on $T_I$, color it blue and move this blue mass  to $T_{I^{k_I+10}}$. No measure then is left in $T_I$. All measure movements are ``up''. It never happens that measure is moved into square $Q_J$, $J\in\cD(\om)$, from outside of $Q_J$. Therefore,
new measure satisfies the same Carleson condition \eqref{1bC} for all boxes $Q_J$, where $J$ is in this $\cD(\om)$.

Otherwise we call $T_I$ bad, we color it white. Do nothing else. 

Then we look at intervals of size $2^{-N+1}$ and repeat all that. We do this for every $\cD(\om)$. Obviously the same $T_I$ can be good for some $\om$ and bad for others. We established above that the probability to be good is at  least $1/2$.

It may happen that a certain $T_J$ has blue mass (moved from below) and original mass. If we need to move mass from $T_J$ we color blue and move only original mass, the  ``new" mass, the blue mass, which came from below, rests unmoved.

When we finish the procedure we have a new measure, and we color it all blue (many parts of it are already colored blue), and we call it $\nu_b(\om)$ (it is random, and it also depends on $f$). But it is dyadic Carleson  like \eqref{1bC} for all boxes $Q_I, I\in \cD(\om)$.

After this procedure, it may very well happen  that for a given $\cD(\om)$ and $J\in \cD(\om)$,  $T_J$ is colored white, but $J$ is colored red and $T_J$ contains blue mass particles.

For every $\om$ we also have subdomains $R$ (colored red) and $W$ (colored white) of $Q_0$, $W=W(\om)$ consisting of bad $T_I, I\in \cD(\om)$, and $R=R(\om)$, consisting of good $T_I, I\in \cD(\om)$.
Now
\begin{equation}\notag
\begin{split}
& \int_R [Pf]^2 d\nu \lesssim \!\!\!\!\!  \!\!\!\!\! \sum_{I\in \cD(\om), T_I \, good} \La f\Ra_{I^{k_I+10}}^2 \nu_b(T_{I^{k_I+10}})+   \La f\Ra_{I_0} ^2|\nu|
\\
&\le  \!\!\!\!\! \!\!\!\!\!  \sum_{J\in \cD(\om), J\, red} \La f\Ra_J^2 \nu_b(T_J)\le C\int_\bR f^2 dx,
\end{split}
\end{equation}

This is because we always preserve dyadic box Carleson \eqref{1bC} property for $\nu_b(\om)$ in corresponding $\cD(\om)$.
On the other hand, let us denote by $\cF$ the union of all $I$'s in all dyadic lattices $\cD(\om), \om \in \Om$, such that $2^{-N} \le |I| \le 2^4$. Then 
$$
\int_\Om \int_{R(\om)} [Pf]^2 d\nu \,d\bP(\om) \ge \frac12  \int_{Q_0} [Pf]^2 d\nu,
$$
because each $T_I$, $I\in \cF$,   will be red at least half of the time (meaning that $\bP\{T_I \, \text{is}\, \text{red}\}\ge 1/2$).

\bigskip

\subsection{Multi-dimensional case.}
\label{2}
Now measure $\nu$ is in $Q_0^n$. We will consider for brevity only the case $n=2$. Measure $\nu$ satisfies Chang--Carleson condition. Let us recall it.
for any open set $G\subset Q_0$, consider its tent: $T_G= (\{z, w) \in \bC_+: R(z, w) \subset G\}$, where 
$$
R(z, w) := [\Re z-\Im z, \Re z +\Im z] \times [\Re w-\Im w, \Re w +\Im w] \,.
$$
Chang--Carleson condition is
\begin{equation}
\label{2CC}
\nu(T_G)  \le C_1\,|G|\,,
\end{equation}
where $|G|$ denotes plane Lebesgue measure of $G$.

As we have Harnack inequality we always may assume that $\nu$ is a doubling measure in the natural metric of $\bC_+^2$.

This allows us to notice the following. Consider any system of dyadic rectangles. Choose any finite family of dyadic rectangles  $R=I\times J$ of this system, we call their union $O$ ``a dyadic open set". It has a dyadic tent $T_O^d$. Now, by definition, it is the union of  all $T_Q$ for all dyadic $Q$ (of the same system) such that $Q\subset O$.

The doubling property above (which we assume without loss of generality because of Harnack's principle) allows us to conclude that if $\nu$ has property \eqref{2CC} it also has the following dyadic Chang--Carleson property:
\begin{equation}
\label{2CCd}
\nu(T_O^d)  \le C\,|O|\,.
\end{equation}

Now let $\cP=P^1P^2$ be the bi-Poisson  extension. Fix a test function $f\ge 0$ supported in $[-1/2, 1/2]^2$. Consider two dyadic lattices of one variable
as before $\cD(\om_x)$, $\cD(\om_y)$, and consider the system of dyadic rectangles $R= I\times J$, $I\in \cD(\om_x), J\in \cD(\om_y)$. Call this system $\cD(\om)$, $\om:=(\om_x, \om_y)$. Let $c_R=(c_I, c_J)$, where $c_I$ is the center of $T_I$,  $c_J$ is the center of $T_J$.

Let $P_I^1$ be  the Poisson kernel with pole at $c_I$, let $P_J^2$ be the Poisson kernel with pole at $c_J$. Bi-Poisson extension $P_I^1P_J^2 $ is bounded by an absolute constant times the convex combination of averages $\La f\Ra_{2^k \,I\times 2^m \, J}$,  $k=0, \dots, \log\frac1{|I|}$,  $m=0, \dots, \log\frac1{|J|}$ and average $\La f\Ra_{Q_0}$, we can  choose $k_I, m_J$ that gives the maximum to $\La f\Ra_{2^k \,I\times 2^m \, J}$,  $k=0, \dots, \log\frac1{|I|}$, $m=0, \dots, \log\frac1{|J|}$, and then we have
$$
\cP f (c_I, c_J)=P_I^1P_J^2 f \le A_1 \La f\Ra_{2^{k_I} \,I\times 2^m \, J} + A_2  \La f\Ra_{Q_0}\,.
$$
Again we can ensure that
\begin{equation}
\label{good2}
\bP \Big\{ 2^{k_I}I\subset I^{k_I+10}, 2^{m_J}I\subset J^{m_J+10}\Big\}\ge 1/4\,.
\end{equation}
Then we just repeat the coloring scheme from subsection \ref{1}. This time we color the $4D$ rectangles $T_R$, $R\in \cD(\om)$ white, if $T_R$ is bad, namely, if the event in \eqref{good2} did not happen, and color it red and call it good if that event does happen. From red $T_R$ we scoop all the measure $\nu$, color its particles blue and move to $T_{\hat R}$ for the ancestor $\hat R:=I^{k_I+10}\times J^{m_J+10}$ of $R=I\times J$. 

Again we will have that random blue measure $\nu(\om)$ satisfies \eqref{2CCd} as the original measure $\nu$ does. Then we repeat the calculation of subsection \ref{1}.  We should prove the embedding
$$
\int\int_{Q_0\times Q_0} [P^1P^2 f]^2 d\nu \le C\int\int_{I_0\times I_0} f^2 dm_2\,.
$$
We just repeat the averaging over probability calculation of subsection \ref{1}.

\section{Surrogate maximum principle}
\label{T}
From now on our paper is devoted only to the multi-tree case (dyadic $n$-rectangles case). We will need to overcome a major difficulty: the potential theory on multi-trees does not allow maximal principle.

Let $\mu$ be a positive function on an $n$-tree $T^{n}$.
Its \emph{energy} is defined as
\[
\calE[\mu]
:=
\int w (\bfI^*\mu)^2.
\]
We view the weight  $w: T^n\to [0, \infty)$ as fixed, and keep it implicit in the notation.
The energy can be written in terms of the \emph{potential}
\[
\bfV^{\mu} := \bfI( w \bfI^{*} \mu)
\]
as $\calE[\mu] = \int_{T^{n}} \bfV^\mu \dif\mu$.
Consider the truncated potential and energy
\[
\bfV_\delta^{\mu} := \bfI ( \one_{\bfV^{\mu} \leq \delta} w \bfI^{*} \mu ),
\quad
\calE_\delta[\mu]
:=
\int_{T^{n}} \bfV_\delta^{\mu} \dif\mu
=
\int_{\Set{\bfV^{\mu} \leq \delta}} w (\bfI^{*} \mu )^2.
\]
On a $1$-tree, we have the maximum principle
\begin{equation}
\label{eq:maximum-principle}
\bfV_{\delta}^{\mu} \leq \delta.
\end{equation}
It follows that, for any positive function $\rho$ on $T$, we have
\begin{equation}
\label{eq:s-max}
\begin{split}
\int_{T} \bfV_{\delta}^{\mu}\dif \rho
&=
\int_{\Set{\bfV^{\mu} \leq \delta}} w I^{*}\mu I^{*}\rho
\\ &\leq
\min( \delta \abs{\rho}, \calE[\mu]^{1/2}\calE[\rho]^{1/2})
\\ &\leq
(\delta \abs{\rho})^{\kappa} (\calE[\mu]\calE[\rho])^{(1-\kappa)/2}
\end{split}
\end{equation}
for every $\kappa \in (0,1]$, where
\[
\abs{\rho} := \int_{T} \rho.
\]
A similar estimate on $2$-trees, with a specific $\kappa$, was obtained in \cite{AMPVZ}.
In this section, we give a streamlined proof of such an estimate on $2$-trees and extend it to $3$-trees.

\bigskip

We do not know how to deal with $n$-trees with $n\ge 4$.

\bigskip

If $T^{n} = T_{1} \times \dotsm \times T_{n}$ is an $n$-tree, then we denote by $I_{1},\dotsc,I_{n}$ the Hardy operators acting in the respective coordinates, so that $\bfI = I_{1}\dotsm I_{n}$.
We use a similar index convention for operators $\Delta_{1},\dotsc,\Delta_{n}$.

\subsection{$1$-trees}
\begin{lemma}
\label{lem:split}
Let $T$ be a tree and $f, g:T\to  [0, \infty)$ be any functions.
Then
\[
(If)(Ig) \leq I(If \cdot g+f\cdot Ig).
\]
\end{lemma}
\begin{proof}
\begin{align*}
If(\alpha) Ig(\alpha)
&\leq
If(\alpha) Ig(\alpha) + I(fg)(\alpha)
\\ &=
\sum_{\alpha'\geq \alpha, \alpha'' \geq \alpha} f(\alpha') g(\alpha'') + \sum_{\alpha'\geq \alpha} f(\alpha') g(\alpha')
\\ &=
\sum_{\alpha'\geq\alpha''\geq\alpha} f(\alpha') g(\alpha'') + \sum_{\alpha''\geq\alpha'\geq\alpha} f(\alpha') g(\alpha'')
\\ &=
\sum_{\alpha''\geq\alpha} If(\alpha'') g(\alpha'') + \sum_{\alpha'\geq\alpha} f(\alpha') Ig(\alpha')
\\ &=
I(If \cdot g)(\alpha) + I(f \cdot Ig)(\alpha).
\qedhere
\end{align*}
\end{proof}

\begin{definition}
\label{def:superadditive}
Given a finite tree $T$, the set of \emph{children} of a vertex $\beta\in T$ consists of the maximal elements of $T$ that are strictly smaller than $\beta$:
\[
\ch \beta := \max \Set{ \beta' \in T \colon \beta' < \beta}
\]
A function $g : T \to \bbR$ is called \emph{superadditive} if for every $\beta \in T$ we have
\[
g(\beta) \geq \sum_{\beta' \in \ch(\beta)} g(\beta').
\]
The difference operator is defined by
\[
\Delta g(\beta) := g(\beta) - \sum_{\beta' \in \ch(\beta)} g(\beta').
\]
\end{definition}

\begin{lemma}[Partial summation]
\label{lem:partial-summation}
Let $T$ be a finite tree.
For any functions $f,g : T \to \bbR$, we have
\begin{equation}
\label{eq:partial-summation}
\sum_{\alpha\in T} f(\alpha) g(\alpha) = \sum_{\alpha'\in T} \Delta f(\alpha') Ig(\alpha').
\end{equation}
\end{lemma}
\begin{proof}
By induction on the size of the tree, one can show
\[
f(\alpha) = \sum_{\alpha' \leq \alpha} \Delta f(\alpha').
\]
It follows that
\[
\sum_{\alpha} f(\alpha) g(\alpha)
=
\sum_{\alpha,\alpha' : \alpha'\leq\alpha} \Delta f(\alpha') g(\alpha)
=
\sum_{\alpha'} \Delta f(\alpha') \sum_{\alpha : \alpha'\leq \alpha}g(\alpha)
=
\sum_{\alpha'\in T} \Delta f(\alpha') Ig(\alpha').
\qedhere
\]
\end{proof}

\begin{lemma}
\label{lem:1}
Let $T$ be a tree and $f,g : T \to \bbR$.
Then
\[
I^{*}(fg)
=
I^{*}(\Delta f \cdot Ig) - f (Ig-g).
\]
\end{lemma}
\begin{proof}
For $\beta \in T$, write $\downset \beta := \Set{\alpha \in T \given \alpha \leq \beta}$.
This is again a sub-tree, on which we can apply the partial summation identity \eqref{eq:partial-summation}.
Hence,
\[
I^{*}(fg)(\beta)
=
\int_{\downset\beta} fg
=
\int_{\downset\beta} \Delta f \cdot I(g \one_{\downset\beta})
\]
For each $\alpha \in \downset\beta$, we have
\[
I(g \one_{\downset\beta})(\alpha)
=
\sum_{\gamma : \alpha \leq \gamma \leq \beta} g(\gamma)
=
\sum_{\gamma : \alpha \leq \gamma} g(\gamma) - \sum_{\gamma : \beta \leq \gamma} g(\gamma) + g(\beta)
=
Ig(\alpha) - Ig(\beta) + g(\beta).
\]
Therefore,
\begin{align*}
I^{*}(fg)(\beta)
&=
\int_{\downset\beta} \Delta f \cdot (Ig - Ig(\beta) + g(\beta))
\\ &=
\int_{\downset\beta} \Delta f \cdot Ig
-
(Ig(\beta) - g(\beta)) \int_{\downset\beta} \Delta f
\\ &=
I^{*}(\Delta f \cdot Ig)(\beta) - (Ig(\beta) - g(\beta)) f(\beta).
\qedhere
\end{align*}
\end{proof}

\begin{corollary}[{cf.~\cite[Lemma 2.2]{AMPVZ}}]
\label{cor:1}
Let $T$ be a tree and $f, g:T\to  [0, \infty)$ .
Then
\[
I^{*}(fg)
\leq
I^{*}(\Delta f \cdot Ig).
\]
\end{corollary}

\subsection{$2$-trees}
In this section we prove a version of \eqref{eq:s-max} on $2$-trees that refines \cite[Lemma 4.1]{AMPVZ}.
Recall that $\bbI = I_{1}I_{2}$.
\begin{lemma}
\label{lem:split2}
Let $T^{2}$ be a bi-tree and $f, g :T^2\to  [0, \infty)$ 
Then
\[
(\bbI f)(\bbI g)
\leq
\bbI (\bbI f \cdot g + I_{1}f \cdot I_{2}g + I_{2}f \cdot I_{1}g + f \cdot \bbI g).
\]
\end{lemma}

\begin{proof}
The linear operators $I_{1},I_{2}$ commute and $\bbI=I_{1}I_{2}$.
To each of $I_{1},I_{2}$ we can apply Lemma~\ref{lem:split}.
Hence,
\begin{align*}
(\bbI f)(\bbI g)
&=
(I_{1}I_{2} f)(I_{1}I_{2} g)
\\ &\leq
I_{1} \Bigl( (I_{1}I_{2} f)(I_{2} g) + (I_{2} f)(I_{1}I_{2} g) \Bigr).
\end{align*}
By Lemma~\ref{lem:split}, the sum in the bracket is
\begin{align*}
&=
(I_{2}I_{1} f)(I_{2} g) + (I_{2} f)(I_{2}I_{1} g)
\\ &\leq
I_{2} \bigl( (I_{2}I_{1} f)(g) + (I_{1} f)(I_{2} g) \bigr)
+ I_{2} \bigl( (I_{2} f)(I_{1}g) + (f)(I_{2}I_{1} g) \bigr).
\end{align*}
Hence,
\begin{align*}
(\bbI f)(\bbI g)
&\leq
I_{1} \Bigl( I_{2} \bigl( (I_{2}I_{1} f)(g) + (I_{1} f)(I_{2} g) \bigr)
+ I_{2} \bigl( (I_{2} f)(I_{1}g) + (f)(I_{2}I_{1} g) \bigr) \Bigr)
\\ &=
\bbI (\bbI f \cdot g + I_{1}f \cdot I_{2}g + I_{2}f \cdot I_{1}g + f \cdot \bbI g).
\qedhere
\end{align*}
\end{proof}

The following result will not be used in our current treatment of bi-trees.
We include it to illustrate the relation of Lemma~\ref{lem:split2} with the argument in \cite{AMPVZ}.
\begin{corollary}[{cf.\ \cite[Theorem 3.1]{AMPVZ}}]
Let $0 < \delta \leq \lambda/4$.
Let $f:T^2\to  [0, \infty)$  with $\supp f \subseteq \Set{\bbI f \leq \delta}$.
Then
\[
(\bbI f) \one_{\bbI f \geq \lambda}
\leq
4\lambda^{-1} \bbI \Bigl( I_{1} f \cdot I_{2}f \Bigr).
\]
\end{corollary}
\begin{proof}
Substituting $f=g$, Lemma~\ref{lem:split2} implies that
\[
(\bbI f)^{2}
\leq
2 \bbI \Bigl(  I_{1} f \cdot I_{2}f + f \cdot \bbI f \Bigr).
\]

Using the support condition, this implies
\begin{align*}
(\bbI f) \one_{\bbI f \geq \lambda}
&\leq
\lambda^{-1} (\bbI f)^{2} \one_{\bbI f \geq \lambda}
\\ &\leq
\lambda^{-1} 2 \bbI \Bigl( I_{1} f \cdot I_{2}f + \delta f \Bigr)
\\ &\leq
2\lambda^{-1} \bbI \Bigl( I_{1} f \cdot I_{2}f \Bigr) + 2\delta\lambda^{-1} \bbI f.
\end{align*}
Since $2\delta\lambda^{-1} \leq 1/2$, this implies
\begin{align*}
(\bbI f) \one_{\bbI f \geq \lambda}
&\leq
4\lambda^{-1} \bbI \Bigl( I_{1} f \cdot I_{2}f \Bigr)
\qedhere
\end{align*}
\end{proof}

\subsubsection{Energy bound}
\begin{lemma}
\label{lem:en2:1}
Let $T^{2}$ be a $2$-tree and $f:T^2\to  [0, \infty)$ a function that is superadditive in each parameter separately.
\hypothesistensor{2}
Suppose that $\supp f \subseteq \Set{ \bbI (wf) \leq \delta }$.
Then
\[
\int_{T^{2}} wf \cdot I_{1}(w_{1}f) \cdot I_{2}(w_{2}f) \cdot \bbI(wf)
\leq
\delta^{2} \int_{T^{2}} w f^{2}.
\]
\end{lemma}
\begin{proof}
By the hypothesis, the left-hand side of the conclusion is
\begin{equation}
\label{eq:6}
\begin{split}
&\leq
\delta \int_{T^{2}} wf \cdot I_{1}(w_{1}f) \cdot I_{2}(w_{2}f)
\\ &=
\delta \int_{T^{2}} w_{1}f \cdot I_{1}(wf) \cdot I_{2}(w_{2}f)
\\ &=
\delta \int_{T^{2}} wf \cdot I_{1}^{*}(w_{1}f \cdot I_{2}(w_{2}f))
\\ &=
\delta \int_{T^{2}} wf \cdot I_{1}^{*}(f \cdot I_{2}(wf))
\end{split}
\end{equation}
By Corollary~\ref{cor:1}, we have
\[
I_{1}^{*}(f \cdot I_{2}(wf))
\leq
I_{1}^{*}(\Delta_{1}f \cdot I_{1}I_{2}(wf)).
\]
Since $\Set{\bbI (wf) \leq \delta}$ is an up-set, $\Delta_{1}f$ is supported on this set.
Since $f$ is superadditive, $\Delta_{1} f\geq 0$.
Hence,
\begin{equation}
\label{eq:4}
I_{1}^{*}(f \cdot I_{2}(wf))
\leq
I_{1}^{*}(\Delta_{1} f \cdot \bbI(wf))
\leq
I_{1}^{*}(\Delta_{1} f \cdot \delta)
=
\delta f.
\end{equation}
Inserting \eqref{eq:4} into \eqref{eq:6}, we obtain the claim.
\end{proof}

\begin{lemma}
\label{lem:en2:2}
Let $T^{2}$ be a $2$-tree and $f:T^2\to  [0, \infty)$ a function that is superadditive in each parameter separately.
\hypothesistensor{2}
Suppose that $\supp f \subseteq \Set{ \bbI (wf) \leq \delta }$.
Then
\[
\int_{T^{2}} w (I_{1}w_{1}f)^{2} (I_{2}w_{2}f)^{2} \leq 4\delta^{2} \int_{T^{2}} w f^{2}.
\]
\end{lemma}

\begin{proof}
By Lemma~\ref{lem:split} and commutativity of operations in different coordinates,
\begin{equation}
\label{eq:3}
\begin{split}
\int_{T^{2}} w(I_{1}w_{1}f)^{2} (I_{2}w_{2}f)^{2}
&\leq
4 \int_{T^{2}} w I_{1}(w_{1}f \cdot I_{1}(w_{1}f)) \cdot I_{2}(w_{2}f \cdot I_{2}(w_{2}f))
\\ &=
4 \int_{T^{2}} I_{1}(w_{1}f \cdot I_{1}(wf)) \cdot I_{2}(w_{2}f \cdot I_{2}(wf))
\\ &=
4 \int_{T^{2}} I_{2}^{*}(w_{1}f \cdot I_{1}(wf)) \cdot I_{1}^{*}(w_{2}f \cdot I_{2}(wf))
\\ &=
4 \int_{T^{2}} w I_{2}^{*}(f \cdot I_{1}(wf)) \cdot I_{1}^{*}(f \cdot I_{2}(wf)).
\end{split}
\end{equation}
Using \eqref{eq:4}, we obtain the claim.
\end{proof}

The next results improve \cite[Lemma 4.1]{AMPVZ}.

\begin{lemma}[Small energy majorization on bi-tree]
\label{SmEMaj2}
Let $T^{2}$ be a $2$-tree and $f:T^2\to  [0, \infty)$ a function that is superadditive in each parameter separately.
\hypothesistensor{2}
Suppose that $\supp f \subseteq \Set{ \bbI (wf) \leq \delta }$. Let $\la \ge 4 \delta$.
Then there exists $\varphi:T^2\to  [0, \infty)$ such that
\[
a)\, \,\bI w\varphi \ge \bI wf, \,\, \text{where}\,\, \bI wf \in [\la, 2\la],
\]
\[
b) \int_{T^2} w \varphi^2 \le C\frac{\delta^2}{\la^2} \int_{T^2} w f^2,
\]
where $C$ is an absolute constant.
\end{lemma}

\begin{proof}
Since $2\delta\lambda^{-1} \leq 1/2$, we have
\begin{align*}
(\bbI f) \one_{\bbI f \geq \lambda}
&\leq
4\lambda^{-1} \bbI \Bigl( I_{1} f \cdot I_{2}f \Bigr)
\end{align*}
And thus
\[
(\bbI f) \one_{\la \le \bbI f \leq 2\lambda} \le 4\lambda^{-1} \bbI \Bigl( I_{1} f \cdot I_{2}f \Bigr) \one_{\la \le \bbI f \leq 2\lambda} \le 4\lambda^{-1} \bbI \Bigl( I_{1} f \cdot I_{2}f \cdot \one_{ \bbI f \leq 2\lambda} \Bigr)
\]
Put
\[
\varphi:= 4\lambda^{-1}  \Bigl( I_{1} f \cdot I_{2}f\cdot {\bf 1}_{\bI f \le 2\la}\Bigr)
\]
Then $\varphi$ does already satisfy  condition a) of the statement of the lemma. Now apply Lemma \ref{lem:en2:2} to see that condition b) of the statement of the lemma is satisfied as well.
\end{proof}

\subsubsection{The lack of maximal principle and the capacity of bad sets}
\label{degree}
In \cite{AMPVZ} we proved the analogous small energy majorization statement  on bi-tree $T^2$ but with $\frac{\delta}{\lambda}$ in the right hand side of b).

Let us see why we care. Let $\mu$ be a measure on $\pd T^2$ and let it potential $\bV^\mu \le 1$ on $\supp\mu$. In the  ``usual'' potential theory the maximal principle would  imply
that potential $\bV^\mu\le 1$ everywhere (or at least that $\bV^\mu\le C$ with absolute constant $C$, see \cite{AH}).  

This is not true for potential theory on multi-trees. The reader can find the counterexamples in \cite{AMPVZ}.

The natural question arises: given $\la>> 1$, what is the size of the set $\{ \bV^\mu \ge \la\}$.  Let us introduce the usual notion of capacity on $T^2$. 
Given a set $E$ we consider all $\varphi$ such that $\bI \vf \ge 1$ on $E$ and 
\[
\text{cap} (E):= \inf\int_{T^2} \vf^2
\]
where infimum is taken over such $\vf$. So one would like to estimate the capacity of the bad set $\text{cap}(\{ \bV^\mu >\la\})$ in terms of  $\la$, if $\bV^\mu \le 1$ on $\supp\mu$.

\begin{theorem}
\label{cap2}
Let  us be on $T^2$ and $\bV^\mu \le 1$ on $\supp \mu$. Then
\[
\text{cap}(\{ \bV^\mu >\la\}) \le \frac{C\mathcal{E}[\mu]}{\la^4}
\] 
for $\la \ge 1$, where $C$ is an absolute constant.
\end{theorem}

\begin{proof}
Consider $f=\bI^* \mu$, $\delta=1$.  If $f(\al) \neq 0$ then there is $\beta\le \al$ such that $\beta\in \supp \mu$. But then by assumption $\bI f(\beta)= \bI\bI^* \mu(\beta) = \bV^\mu(\beta) \le 1$. By monotonicity of $\bI$ we have that $\bI f(\al) \le 1$. Hence
\[
\supp f \subset \{\bI f \le \delta=1\},
\]
and we are in the assumptions of  small energy majorization Lemma  on bi-tree \ref{SmEMaj2}.  We apply it with data $(f, \delta=1, \la:= 2^m\la)$ to get functions
$\vf_m$, $m=0, 1, \dots$ such that
\[
\bI \vf_m  \ge \bI f=\bV^{\mu}, \quad \text{where}\,\,  \bV^{\mu}\in [2^m \la, 2^{m+1}\la],
\]
which means that 
\[
2^{-m} \la^{-1}\bI \vf_m  \ge 1, \quad \text{where}\,\,  \bV^{\mu}\in [2^m \la, 2^{m+1}\la],
\]
On the other hand, putting $\vf :=\sum_m 2^{-m} \la^{-1} \vf_m$, we get
firstly
\[
\bI \vf \ge 1, \quad \text{where}\,\,  \bV^{\mu}\in [ \la, \infty),
\]
and secondly
\begin{equation*}
\begin{split}
&\int\vf^2 \le \Big(\la^{-1}\sum_m 2^{-m} \big(\int_{T^2} \vf_m^2\big)^{1/2}\Big)^2 \le  
\\
&C\,\Big(\la^{-1}\sum_m \la^{-1}2^{-2m}( \int_{T^2}  f^2)^{1/2}\Big)^2 \le C'\, \la^{-4} \int_{T^2}  f^2
\end{split}
\end{equation*}
As $f=\bI^* \mu$, $\int_{T^2} f^2 =  \int_{T^2} \bI^*\mu \bI^* \mu = \int_{T^2} \bI\bI^* \mu d\mu = \int_{T^2} \bV^\mu d\mu =\cE[\mu]$, which proves theorem.
\end{proof}

\begin{remark}
We do not know how precise is the rate $\la^{-4}$ in Theorem \ref{cap2}. We do not even know whether the sharp rate should be polynomial or exponential. What we do know (see \cite{AMPVZ}) is that  for any large $\la$ there exists a measure $\mu$, such that $\bV^\mu \le 1 $ on $\supp  \mu$ but with positive absolute constant $c$ the following holds
\begin{equation}
\label{capBelow}
\text{cap}(\{ \bV^\mu >\la\}) \ge c e^{-2\la}\,.
\end{equation}
\end{remark}

\subsubsection{Continuation of energy estimates}
\label{continEE2}
\begin{lemma}
\label{lem:bVdelta:2}
Let $\mu,\rho$ be positive measures on $T^{2}$ and $\delta>0$.
\hypothesistensor{2}
Then
\begin{equation}
\Bigl( \int \bbV_{\delta}^{\mu} \dif\rho \Bigr)^{4}
\leq
28 \cdot \delta^{2} \calE_\delta[\mu] \calE[\rho] \abs{\rho}^{2}.
\end{equation}
\end{lemma}

\begin{proof}
Let $f := \one_{\bbV^{\mu} \leq \delta} \bbI^{*}\mu$.
Then
\begin{align*}
\int \bbV_{\delta}^{\mu} \dif\rho
&=
\int \bbI (wf) \dif\rho
\\ &\leq
\abs{\rho}^{1/2} \Bigl( \int (\bbI (wf))^{2} \dif\rho \Bigr)^{1/2}
\intertext{by Lemma~\ref{lem:split2},}
&\leq
\abs{\rho}^{1/2} \Bigl( 2 \int \bbI(I_{1}(wf) \cdot I_{2}(wf) + (wf) \cdot \bbI (wf)) \dif\rho \Bigr)^{1/2}
\\ &=
2^{1/2} \abs{\rho}^{1/2} \Bigl( \int w (I_{1}(w_{1}f) \cdot I_{2}(w_{2}f) + f \cdot \bbI (wf)) \bbI^{*}\rho \Bigr)^{1/2}
\\ &\leq
2^{1/2} \abs{\rho}^{1/2} \calE[\rho]^{1/4} \Bigl( \int w (I_{1}(w_{1}f) \cdot I_{2}(w_{2}f) + f \cdot \bbI (wf))^{2} \Bigr)^{1/4}
\intertext{expanding the square and using Lemma~\ref{lem:en2:1} and Lemma~\ref{lem:en2:2},}
&\leq
2^{1/2} \abs{\rho}^{1/2} \calE[\rho]^{1/4} \Bigl( 7 \delta^{2} \int w f^{2} \Bigr)^{1/4}
\\ &=
28^{1/4} \abs{\rho}^{1/2} \calE[\rho]^{1/4} \delta^{1/2} \calE_{\delta}[\mu]^{1/4}.
\qedhere
\end{align*}
\end{proof}

\subsection{3-trees}
Similarly to Lemma~\ref{lem:split2}, we obtain the following result for $3$-trees.
\begin{lemma}
\label{lem:split3}
Let $T^{3}$ be a $3$-tree and $f, g :T^3\to  [0, \infty)$.
Then
\[
(\bfI f)(\bfI g)
\leq
\bfI \Bigl( \sum_{A \subseteq \Set{1,2,3}} I_{A} f \cdot I_{A^{c}}g \Bigr),
\]
where $I_{A} = \prod_{i\in A}I_{i}$.
\end{lemma}

\begin{corollary}
\label{cor:upper-bd:3}
Let $0 < \delta \leq \lambda/4$.
Let $f:T^3\to  [0, \infty)$ with $\supp f \subseteq \Set{\bfI f \leq \delta}$.
Then
\[
(\bfI f) \one_{\lambda \leq \bfI f \leq 2 \lambda}
\leq
4\lambda^{-1} \bfI \Bigl( \sum_{i \in \Set{1,2,3}} I_{i} f \cdot I_{(i)}f \cdot \one_{\bfI f \leq 2 \lambda} \Bigr),
\]
where $I_{(i)} = \prod_{j\neq i}I_{j}$.
\end{corollary}
\begin{proof}
Substituting $f=g$, Lemma~\ref{lem:split3} implies that
\[
(\bfI f)^{2}
\leq
\bfI \Bigl( 2\sum_{i \in \Set{1,2,3}} I_{i} f \cdot I_{(i)}f + 2f \cdot \bfI f \Bigr).
\]

Using the support condition, this implies
\begin{align*}
(\bfI f) \one_{\lambda \leq \bfI f \leq 2 \lambda}
&\leq
\lambda^{-1} (\bfI f)^{2} \one_{\lambda \leq \bfI f \leq 2 \lambda}
\\ &\leq
\lambda^{-1} \bfI \Bigl( 2\sum_{i \in \Set{1,2,3}} I_{i} f \cdot I_{(i)}f + 2\delta f \Bigr)
\\ &\leq
\lambda^{-1} \bfI \Bigl( 2\sum_{i \in \Set{1,2,3}} I_{i} f \cdot I_{(i)}f \Bigr) + 2\delta\lambda^{-1} \bfI f.
\end{align*}
Since $2\delta\lambda^{-1} \leq 1/2$, this implies
\begin{align*}
(\bfI f) \one_{\lambda \leq \bfI f \leq 2 \lambda}
&\leq
2\lambda^{-1} \bfI \Bigl( 2\sum_{i \in \Set{1,2,3}} I_{i} f \cdot I_{(i)}f \Bigr) \one_{\lambda \leq \bfI f \leq 2 \lambda}
\\ &\leq
2\lambda^{-1} \bfI \Bigl( 2\sum_{i \in \Set{1,2,3}} I_{i} f \cdot I_{(i)}f \cdot \one_{\bfI f \leq 2 \lambda} \Bigr).
\qedhere
\end{align*}
\end{proof}

\subsubsection{Energy bound}
\begin{lemma}
\label{lem:energy-bd:3}
Let $f:T^3\to  [0, \infty)$ be superadditive.
Let $w:T^3\to  [0, \infty)$ be a tensor product.
Suppose that $\supp f \subseteq \Set{\bfI (wf) \leq \delta}$.
Then
\[
\int w (I_{1}(w_{1}f) \cdot I_{2}I_{3}(w_{2}w_{3}f))^{2} \one_{\bfI (wf) \leq \lambda}
\leq
2\delta\lambda \int w f^{2}.
\]
\end{lemma}
\begin{proof}
By Lemma~\ref{lem:split}, we have
\begin{equation}
\label{eq:5}
\begin{split}
\MoveEqLeft
\int w (I_{1}(w_{1}f) \cdot I_{2}I_{3}(w_{2}w_{3}f))^{2} \one_{\bfI (wf) \leq \lambda}
\\ &\leq
2 \int w I_{1}(w_{1}f \cdot I_{1}(w_{1}f)) \cdot (I_{2}I_{3}(w_{2}w_{3}f))^{2} \one_{\bfI (wf) \leq \lambda}
\\ &=
2 \int I_{1}(w_{1}f \cdot I_{1}(wf)) \cdot (I_{2}I_{3}(w_{2}w_{3}f)) \cdot (I_{2}I_{3}(wf)) \one_{\bfI (wf) \leq \lambda}
\\ &=
2 \int w_{1}f \cdot I_{1}(wf) \cdot I_{1}^{*}\bigl( (I_{2}I_{3}(w_{2}w_{3}f)) \cdot (I_{2}I_{3}(wf)) \one_{\bfI (wf) \leq \lambda} \bigr).
\end{split}
\end{equation}
By Corollary~\ref{cor:1}, we have
\[
I_{1}^{*}\bigl( (I_{2}I_{3}(w_{2}w_{3}f)) \cdot (I_{2}I_{3}(wf)) \one_{\bfI (wf) \leq \lambda} \bigr)
\leq
I_{1}^{*}\bigl( \Delta_{1} (\one_{\bfI (wf) \leq \lambda} \cdot I_{2}I_{3}(w_{2}w_{3}f)) \cdot I_{1}(I_{2}I_{3}(wf)).
\]
Since $\Set{\bfI (wf) \leq \lambda}$ is an up-set and $f$ is superadditive in the first coordinate, we have $\Delta_{1} (\one_{\bfI (wf) \leq \lambda} \cdot I_{2}I_{3}(w_{2}w_{3}f)) \geq 0$, and $I_{1}(I_{2}I_{3}wf) = \bfI wf \leq \lambda$ on the support of the former function.
Hence,
\begin{align*}
I_{1}^{*}\bigl( (I_{2}I_{3}(w_{2}w_{3}f)) \cdot (I_{2}I_{3}(wf)) \one_{\bfI (wf) \leq \lambda} \bigr)
&\leq
I_{1}^{*}\bigl( \Delta_{1} (\one_{\bfI (wf) \leq \lambda} \cdot I_{2}I_{3}(w_{2}w_{3}f)) \cdot \lambda \bigr)
\\ &=
\lambda \one_{\bfI (wf) \leq \lambda} \cdot I_{2}I_{3}(w_{2}w_{3}f).
\end{align*}
Using this bound, we obtain
\begin{align*}
\eqref{eq:5}
&\leq
2 \lambda \int w_{1}f \cdot I_{1}(wf) \cdot I_{2}I_{3}(w_{2}w_{3}f)
\\ &=
2 \lambda \int f \cdot I_{1}(wf) \cdot I_{2}I_{3}(wf)
\\ &=
2 \lambda \int wf \cdot I_{1}^{*}(f \cdot I_{2}I_{3}(wf)).
\end{align*}
As in \eqref{eq:4}, we see that
\[
I_{1}^{*}(f \cdot I_{2}I_{3}(wf))
\leq
\delta f.
\]
This implies the conclusion of the lemma.
\end{proof}
Compare the next result with Lemma \ref{SmEMaj2}.
\begin{lemma}[Small energy majorization on tri-tree]
\label{SmEMaj3}
Let $T^{3}$ be a $3$-tree and $f:T^3\to  [0, \infty)$ a function that is superadditive in each parameter separately.
\hypothesistensor{3}
Suppose that $\supp f \subseteq \Set{ \bfI (wf) \leq \delta }$. Let $\la \ge 4 \delta$.
Then there exists $\varphi:T^3\to  [0, \infty)$ such that
\[
a)\, \,{\bf I} (w\varphi )\ge \bfI (wf), \,\, \text{where}\,\, \bfI (wf) \in [\la, 2\la],
\]
\[
b) \int_{T^3} w \varphi^2 \le C\frac{\delta}{\la} \int_{T^3} w f^2,
\]
where $C$ is an absolute constant.
\end{lemma}

\begin{proof}
Since $2\delta\lambda^{-1} \leq 1/2$, we have
\begin{align*}
(\bfI wf) \one_{\lambda \leq \bfI f \leq 2 \lambda}
&\leq
2\lambda^{-1} \bfI \Bigl( 2\sum_{i \in \Set{1,2,3}} I_{i}w_{i} f \cdot I_{(i)}w_{(i)}f \Bigr) \one_{\lambda \leq \bfI f \leq 2 \lambda}
\\ &\leq
2\lambda^{-1} \bfI \Bigl( 2\sum_{i \in \Set{1,2,3}} I_{i}w_{i} f \cdot I_{(i)}w_{(i)}f \cdot \one_{\bfI f \leq 2 \lambda} \Bigr).
\end{align*}
Put
\[
\vf:= 2\lambda^{-1} \Bigl( 2\sum_{i \in \Set{1,2,3}} I_{i} w_{i}f \cdot I_{(i)}w_{(i)}f \cdot \one_{\bfI f \leq 2 \lambda} \Bigr).
\]
Then we have just seen that a) is satisfied. To prove b) just apply Lemma \ref{lem:energy-bd:3}.
 
\end{proof}

\subsubsection{The lack of maximal principle and the capacity of bad sets}
\label{degree3}
The reader can compare this subsection with  Subsection \ref{degree}.

Let $\mu$ be a measure on $\pd T^3$ and let it potential $\bfV^\mu \le 1$ on $\supp\mu$. As we already mentioned in the  ``usual'' potential theory the maximal principle would  imply
that potential $\bV^\mu\le 1$ everywhere (or at least that $\bfV^\mu\le C$ with absolute constant $C$, see \cite{AH}).  

As we also already mentioned, see Subsection \ref{degree}, this is not true for potential theory on multi-trees. 

The natural question arises: given $\la>> 1$, what is the size of the set $\{ \bfV^\mu \ge \la\}$.  Let us introduce the usual notion of capacity on $T^3$. 
Given a set $E$ we consider all $\varphi$ such that $\bfI \vf \ge 1$ on $E$ and 
\[
\text{cap} (E):= \inf\int_{T^3} \vf^2
\]
where infimum is taken over such $\vf$. So one would like to estimate the capacity of the bad set $\text{cap}(\{ \bfV^\mu >\la\})$ in terms of  $\la$, if $\bfV^\mu \le 1$ on $\supp\mu$.

\begin{theorem}
\label{cap3}
Let  us be on $T^3$ and $\bfV^\mu \le 1$ on $\supp \mu$. Then
\[
\text{cap}(\{ \bV^\mu >\la\}) \le \frac{C\mathcal{E}[\mu]}{\la^3}
\] 
for $\la \ge 1$, where $C$ is an absolute constant.
\end{theorem}

\begin{proof}
Consider $f=\bfI^* \mu$, $\delta=1$.  If $f(\al) \neq 0$ then there is $\beta\le \al$ such that $\beta\in \supp \mu$. But then by assumption $\bfI f(\beta)= \bfI\bfI^* (\beta) = \bfV^\mu(\beta) \le 1$. By monotonicity of $\bfI$ we have that $\bfI f(\al) \le 1$. Hence
\[
\supp f \subset \{\bfI f \le \delta=1\},
\]
and we are in the assumptions of  small energy majorization Lemma  on tri-tree \ref{SmEMaj3}.  We apply it with data $(f, \delta=1, \la:= 2^m\la)$ to 
finish the proof in exactly the same manner as this has been done in Theorem \ref{cap2}.

\end{proof}

\begin{remark}
We do not know how precise is the rate $\la^{-2}$ in Theorem \ref{cap3}.  It is obviously worse than the one on bi-tree, but we do not know how sharp it is.
\end{remark}

\subsubsection{Continuation of energy estimates}
\label{continEE3}
The next result is a version of \eqref{eq:s-max} for $3$-trees.
The proof closely follows \cite[Lemma 4.1]{AMPVZ}.

\begin{lemma}
\label{lem:bVdelta:3}
Let $\mu,\rho$ be positive measures on $T^{3}$ and $\delta>0$.
\hypothesistensor{3}
Then
\begin{equation}
\label{eq:bVdelta:3}
\Bigl( \int \bfV_{\delta}^{\mu} \dif\rho \Bigr)^{3}
\lesssim
\delta \calE_\delta[\mu] \calE[\rho] \abs{\rho}.
\end{equation}
\end{lemma}

\begin{proof}
Without loss of generality, $\calE_{\delta}[\mu]\neq 0$ and $\rho\not\equiv 0$.
Let $\lambda>0$ be chosen later.

Let $f := \bfI^{*}\mu \cdot \one_{\bfV^{\mu} \leq \delta}(\alpha)$.
This function is superadditive.
Also, $\bfI (wf) = \bfV_{\delta}^{\mu} \leq \bfV^{\mu} \leq \delta$ on $\supp f$, and $\calE_{\delta}[\mu] = \int w f^{2}$.

For $m=0,1,\dotsc$ let
\[
\phi_{m} := 4(2^{m}\lambda)^{-1} \Bigl( \sum_{i \in \Set{1,2,3}} I_{i} (w_{i}f) \cdot I_{(i)}(w_{(i)}f) \cdot \one_{\bfI (wf) \leq 2^{m+1} \lambda} \Bigr).
\]
Then, by Corollary~\ref{cor:upper-bd:3} with $wf$ in place of $f$, we have
\[
\bfI (wf) \cdot \one_{2^{m}\lambda < \bfI (wf) \leq 2^{m+1}\lambda} \leq \bfI(w\phi_{m}),
\]
and, by Lemma~\ref{lem:energy-bd:3}, we have
\[
\int w \phi_{m}^{2} \lesssim \frac{\delta}{2^{m}\lambda} \int w f^{2}.
\]
Hence,
\begin{align*}
\int \bfV^\mu_\delta \dif\rho
&=
\int_{\{\bfV^\mu_{\delta} \leq \lambda\}} \bfV^\mu_{\delta} \dif\rho
+ \sum_{m=0}^{\infty} \int_{\{2^{m}\lambda < \bfV_{\delta}^{\mu} \leq 2^{m+1}\lambda\}} \bfV^\mu_\delta \dif\rho
\\ &\le
\lambda \abs{\rho}
+ \sum_{m=0}^{\infty} \int \bfI (w\phi_{m}) \dif\rho
\\ &=
\lambda \abs{\rho}
+ \sum_{m=0}^{\infty} \int w \phi_{m} \bfI^{*}\dif\rho
\\ &\leq
\lambda \abs{\rho}
+ \sum_{m=0}^{\infty} \bigl( \int w \phi_{m}^{2} \bigr)^{1/2} \calE[\rho]^{1/2}
\\ &\leq
\lambda \abs{\rho}
+ \sum_{m=0}^{\infty} C (\delta/(2^{m}\lambda))^{1/2} \calE_{\delta}[\mu]^{1/2} \calE[\rho]^{1/2}.
\\ &\leq
\lambda \abs{\rho}
+ C (\delta/\lambda)^{1/2} \calE_{\delta}[\mu]^{1/2} \calE[\rho]^{1/2}.
\end{align*}
Substituting $\lambda = (\delta \calE_{\delta}[\mu] \calE[\rho])^{1/3} \abs{\rho}^{-2/3}$, we obtain \eqref{eq:bVdelta:3}.
\end{proof}

\begin{corollary}
\label{cor:bVdelta:3}
Let $\mu,\rho$ be positive measures on $T^{3}$ and $\delta>0$.
Then
\begin{equation}
\label{eq:bVdelta:3'}
\int \bfV_{\delta}^{\mu} \dif\rho
\leq
C_{\eqref{eq:bVdelta:3}}^{1/2} \delta^{1/2} \calE[\mu]^{1/6} \abs{\mu}^{1/6} \calE[\rho]^{1/3} \abs{\rho}^{1/3}.
\end{equation}
\end{corollary}
\begin{proof}
By Lemma~\ref{lem:bVdelta:3} and Theorem~\ref{partialV}, we have
\[
\Bigl( \int \bfV_{\delta}^{\mu} \dif\rho \Bigr)^{3}
\leq
C_{\eqref{eq:bVdelta:3}} \delta \calE_\delta[\mu] \calE[\rho] \abs{\rho}
\leq
C_{\eqref{eq:bVdelta:3}} \delta \Bigl( C_{\eqref{eq:bVdelta:3}} \delta \calE[\mu] \abs{\mu} \Bigr)^{1/2} \calE[\rho] \abs{\rho}
\]
\end{proof}

\subsection{$n$-trees}
\label{cond:max-surrogate}
We say that a weight $w$ satisfies the \emph{surrogate maximum principle} if, for some $\kappa>0,C<\infty$ and every positive functions $\mu, \rho: T^n \to [0,\infty)$ and $\delta>0$, we have
\begin{equation}
\label{eq:max-surrogate}
\int \bfV_{\delta}^{\mu} \dif\rho
\leq C
\bigl(\delta \abs{\rho} \bigr)^{\kappa} \bigl( \calE_\delta[\mu] \calE[\rho] \bigr)^{(1-\kappa)/2}.
\end{equation}

When $n \in \Set{1,2,3}$, every weight $w$ of tensor product form satisfies the surrogate maximum principle with $\kappa=1/n$ and $C$ independent of $w$.
For $n=1$, this follows from the maximum principle \eqref{eq:maximum-principle}.
For $n=2$ this holds by Lemma~\ref{lem:bVdelta:2}, and for $n=3$ by Lemma~\ref{lem:bVdelta:3}.
This leads us to the following conjecture.

\begin{conjecture}[Surrogate maximum principle]
\label{conj:max-surrogate}
Let $w : T^n \to [0, \infty)$ be of tensor product form.
Then $w$ satisfies the surrogate maximum principle with $\kappa=1/n$ and $C=C(n)$ independent of $w$.
\end{conjecture}

In what follows, we will work conditionally on the surrogate maximum principle.
All implicit constants are allowed to depend on $\kappa,C$ in \eqref{eq:max-surrogate}, but not otherwise on $w$.
In particular, our results hold unconditionally for $w$ of tensor product form if $n\in\Set{1,2,3}$.

Taking $\rho=\mu$ in \eqref{eq:max-surrogate}, we obtain Lemma~\ref{partialV} below.
\begin{lemma}
\label{partialV}
Let $w:T^n\to [0, \infty)$ be such that the surrogate maximal principle \eqref{eq:max-surrogate} holds.
Let $\mu$ be a positive measure on $T^{n}$ and $\delta>0$.
Then
\begin{equation}
\label{eq:cEcE}
\int \bfV_{\delta}^{\mu} \dif\mu
\leq C_{\eqref{eq:max-surrogate}}^{\frac{2}{1+\kappa}}
(\delta \abs{\mu})^{\frac{2\kappa}{1+\kappa}} \calE[\mu]^{\frac{1-\kappa}{1+\kappa}}.
\end{equation}
\end{lemma}
When using Lemma~\ref{partialV}, we can also denote $\frac{2\kappa}{1+\kappa}$ by the letter $\kappa$ again, which proves \eqref{eq:max-surrogate} for $n=1, 2, 3$.

\begin{conjecture}
\label{123}
For all  positive integers $n$
\begin{equation}
\label{eq:123}
\int \bfV_{\delta}^{\mu} \dif\mu
\leq C_n
(\delta \abs{\mu})^{\frac{2}{n+1}} \calE[\mu]^{\frac{n-1}{n+1}}.
\end{equation}
\end{conjecture}

\section{Carleson condition implies hereditary Carleson condition}

For an arbitrary set $E \subseteq T^{n}$, let
\[
\cE_E[\mu] :=\int_{E} w(\bfI^{*}\mu)^2.
\]
\begin{lemma}
\label{l:l2}
Let $w:T^n\to [0, \infty)$ be such that the surrogate maximal principle \eqref{eq:max-surrogate} holds.
Let $\nu: T^n \to [0,\infty)$ and
\begin{equation}
\label{eq:large-energy-downset}
E := \Set[\Big]{ \bfV^{\nu} > (2 C_{\eqref{eq:cEcE}})^{-1/\kappa} \frac{\calE[\nu]}{\abs{\nu }} } \subseteq T^{3}.
\end{equation}
Then
\begin{equation}\label{e:59}
\calE_{E}[\nu]
:=
\sum_{\alpha \in E} w(\alpha)(\bfI^{*}\nu(\alpha))^{2}
\geq
\frac{1}{2} \calE[\nu].
\end{equation}
\end{lemma}

\begin{proof}
Put $\delta := (2 C_{\eqref{eq:cEcE}})^{-1/\kappa} \frac{\calE[\nu]}{\abs{\nu}}$.
By Lemma~\ref{partialV}, we have
\[
\calE_{E}[\nu]
=
\calE[\nu] - \calE_{\delta}[\nu]
\geq
\calE[\nu] - C_{\eqref{eq:cEcE}} (\delta \abs{\nu})^{\kappa} \calE[\nu]^{1-\kappa}
=
\calE[\nu]/2,
\]
and the claim follows.
\end{proof}

\begin{theorem}
\label{CtoREC}
Let $w:T^n\to [0, \infty)$ be such that the surrogate maximal principle \eqref{eq:max-surrogate} holds.
Then, for every $\mu:T^n\to [0, \infty0$, we have
\[
[w,\mu]_{HC} \lesssim [w,\mu]_{Car}.
\]
\end{theorem}

\begin{proof}
Without loss of generality $[w,\mu]_{Car}=1$.
Let
\begin{equation}
\label{eq:hereditary-const}
A := [w,\mu]_{HC}
=
\sup_{E \subseteq T^{n}, \mu(E) \neq 0} \frac{\calE[\mu\one_{E}]}{\mu(E)}.
\end{equation}
Since $T^{n}$ is finite, the constant $A$ is finite, and there exists a maximizer $E$ for \eqref{eq:hereditary-const}.
Let $\nu := \mu\one_{E}$ and
\begin{equation}
\label{e:57}
\calD := \Set[\Big]{ \bfV^{\nu} > c A }
\end{equation}
with a small constant $c$.
Then, by Lemma~\ref{l:l2}, we have
\[
\cE_{\cD}[\nu]\ge \frac 12 \cE[\nu]\,.
\]
Hence, $ 0<\cE[\nu]\le 2\cE_{\cD}[\nu] \le 2\cE_{\cD}[\mu] \le 2\mu(\cD)$.
In particular, $\mu(\cD)\neq 0$.

By definition, we have $\bfV^{\nu} > cA$ on $\cD$, and therefore
\[
cA \mu(\cD)
\le
\int_\cD\bfV^\nu\dif\mu
\le
\cE[\nu]^{1/2} \cE[\mu \one_{\cD}]^{1/2}
\le
(2 \mu(\cD))^{1/2} (A \mu(\cD))^{1/2}.
\]
It follows that $A \lesssim 1$.
\end{proof}

\section{Hereditary Carleson condition implies Carleson embedding}
\label{Hered}

\begin{theorem}
\label{murho}
Let $w:T^n\to [0, \infty)$ be such that the surrogate maximal principle \eqref{eq:max-surrogate} holds.
Let $\mu, \rho$ be positive measures on $T^{n}$ with
\begin{equation}
\label{c}
[w,\mu]_{REC} \leq 1,
\quad
[w,\rho]_{REC} \leq 1.
\end{equation}
Then, for some $\kappa'>0$, we have
\begin{equation}
\label{eq:1}
\int \bfV^{\mu} \dif\rho
\lesssim
\abs{\mu }^{1/2-\kappa'} \abs{\rho }^{1/2+\kappa'}.
\end{equation}
\end{theorem}
\begin{remark}
This improves upon the estimate
\[
\int \bfV^{\mu} \dif\rho
\leq
\calE [\mu]^{1/2} \calE [\rho]^{1/2}
\lesssim
\abs{\mu }^{1/2} \abs{\rho }^{1/2}
\]
that is immediate by Cauchy--Schwarz and the Carleson condition.
\end{remark}

\begin{proof}
Let $\delta > 0$ be chosen later.
By \eqref{eq:max-surrogate} and \eqref{c}, we obtain
\[
\int \bfV_{\delta}^{\mu} \dif\rho
\lesssim
\delta^{\kappa} \abs{\mu}^{(1-\kappa)/2} \abs{\rho}^{(1+\kappa)/2}.
\]

Consider the down-set $E := \Set{\bfV^{\mu} > \delta} \subset T^{n}$.
By the Cauchy--Schwarz inequality and the Carleson condition, we have
\[
\int (\bfV^{\mu}-\bfV_{\delta}^{\mu}) \dif\rho
=
\int_{E} w \bfI^{*}\mu \bfI^{*}\rho
\leq
\calE_{E}[\mu]^{1/2} \calE_{E}[\rho]^{1/2}
\leq
\mu(E)^{1/2} \calE[\rho]^{1/2}.
\]
Note that
\begin{equation}
\label{eq:2}
\delta \mu(E)
\leq
\int_{E} \bfV^{\mu} \dif\mu
\leq
\calE [\mu]^{1/2} \calE [\mu\one_{E}]^{1/2}
\leq
\calE[\mu]^{1/2} \mu(E)^{1/2}
\end{equation}
by definition \eqref{eq:hereditary-Carleson} of the hereditary Carleson constant.
Hence, 
\[
\mu(E)^{1/2} \leq \delta^{-1} \calE[\mu]^{1/2},
\]
and it follows that
\[
\int (\bfV^{\mu}-\bfV_{\delta}^{\mu}) \dif\rho
\leq
\delta^{-1} \calE[\rho]^{1/2} \calE[\mu]^{1/2}.
\]
Hence,
\[
\int \bfV^{\mu} \dif\rho
\leq
C \delta^{\kappa} \abs{\mu}^{(1-\kappa)/2} \abs{\rho}^{(1+\kappa)/2}
+
\delta^{-1} \abs{\rho}^{1/2} \abs{\mu}^{1/2}.
\]
Optimizing in $\delta$, we obtain
\[
\int \bfV^{\mu} \dif\rho
\lesssim
\abs{\mu}^{\frac{1/2}{1+\kappa}} \abs{\rho}^{\frac{1/2+\kappa}{1+\kappa}}.
\qedhere
\]
\end{proof}

Exactly as in \cite[Theorem 6.3]{AMPVZ}, we can now prove the following result.
\begin{theorem}
\label{thm:HC=>CE}
Let $w:T^n\to [0, \infty)$ be such that the surrogate maximal principle \eqref{eq:max-surrogate} holds.
Then, for every $\mu:T^n \to [0, \infty)$, we have
\[
[w,\mu]_{CE} \lesssim [w,\mu]_{HC}.
\]
\end{theorem}


\section{Box condition implies hereditary Carleson}
\label{prep}

\subsection{Main estimate}
\label{s1main}

Define
\begin{align}
\label{eq:V-interval}
\bfV^\nu_P(\om)
&:=
\sum_{Q:\om \le Q\le P} w(Q)\bfI^{*}\nu(Q),
\\
\bfV^\mu_{\eps', good}(\omega)
&:=
\sum_{P \geq \omega: \bfV_P(\om) > \eps'} (w\bfI^{*}\mu)(P).
\end{align}

\begin{lemma}
\label{2main}
Let $n\geq 2$ and $\mu:T^n \to [0, \infty)$.
Let $w:T^n\to [0, \infty)$ be such that the surrogate maximal principle \eqref{eq:max-surrogate} holds.
Assume that $\calE[\mu] \leq \abs{\mu}$ and
\begin{equation}
\label{eq:V-lower-bd}
\bfV^{\mu} \geq 1/3
\quad\text{on } \supp \mu.
\end{equation}
Then, if $\epsilon'$ is small enough, we have
\[
\int \bfV^{\mu}_{\eps', good} \dif \mu
\gtrsim
\abs{\mu}.
\]
\end{lemma}

\begin{proof}[Proof of Lemma~\ref{2main}]
It suffices to show that, for some $\epsilon'$ and $\epsilon_{n-1}$, we have
\[
\mu \Set{ \omega \in T^{n} \given \bfV^{\mu}_{\eps', good}(\omega) \ge \eps_{n-1} }
\geq
\abs{\mu}/2.
\]
Let $\epsilon > 0$ be chosen later and define
\[
\epsilon_{1} := \epsilon,
\quad
\epsilon_{2} := \epsilon\epsilon_{1}^{1/\kappa},
\quad
\epsilon_{3} := \epsilon\epsilon_{2}^{1/\kappa},
\dotsc
\]
By Lemma~\ref{partialV}, we have
\[
\int \bfV^{\mu}_{\epsilon_{j}} \dif\mu
\lesssim
\epsilon_{j}^{\kappa} \abs{\mu}^{\kappa} \calE[\mu]^{1-\kappa}
\lesssim
\epsilon_{j}^{\kappa} \int \dif\mu
\]
for some $\kappa>0$.
By Chebyshov's inequality, it follows that
\begin{equation}
\label{eq:7}
\bfV^\mu_{\eps_{j}}(\om) \le (\epsilon_{j}/\epsilon)^{\kappa}/10
\end{equation}
for a proportion $\ge (1-C\epsilon^{\kappa})$ of $\omega$'s.
So we only consider $\omega$'s for which \eqref{eq:7} holds for all $j=1,\dotsc,n-1$.
Similarly, we may restrict to those $\omega$'s for which $\bfV^{\mu}(\omega) \lesssim 1$.

Let
\[
\epsilon' := \epsilon \cdot \epsilon_{1} \dotsm \epsilon_{n-1}.
\]
For a fixed $\omega$, let
\begin{equation}
\label{eq:calU}
\calU := \Set{ Q\geq \omega \given \bfV_{Q}(\omega) > \epsilon' }
\end{equation}
and
\begin{equation}
\calW_{j} := \Set{ Q\geq \omega \given \bfV^\mu(Q) \le \epsilon_{j} },
\quad
1 \leq j \leq n-1.
\end{equation}
For $p \in T^{n}$, write
\[
\upset p := \Set{\alpha \in T^{n} \given \alpha\geq p}.
\]
For $p \in \upset \omega$, let
\[
\downset p := \Set{\alpha \in T^{n} \given \omega \leq \alpha \leq p}.
\]

If $\calU \not\subseteq \calW_{n-1}$, then this means that there exists $p\not\in \calW_{n-1}$ with $\upset p \subseteq \calU$.
Hence,
\[
\bfV^\mu_{\eps', good}(\omega)
\geq
\sum_{p' \in \upset p} w\mu(p')
=
\bfV^{\mu}(p)
\geq
\epsilon_{n-1}.
\]

Assume now that $\calU \subseteq \calW_{n-1}$.
In this case, we will cover $\upset\omega \setminus \calW_{1}$ by boundedly many sets of the form $\downset q$ with $q \in \upset\omega \setminus \calU$.
This will lead to a contradiction with \eqref{eq:V-lower-bd}, since, by \eqref{eq:7} and \eqref{eq:calU}, the integral of
\[
f := w \bfI^{*}\mu
\]
is small on $\calW_{1}$ and on each such set $\downset q$.

For a set of coordinates $J \subseteq \Set{1,\dotsc,n}$ and a point $p \in T^{n}$, let
\[
\upset_{J} p := \Set{ q \in T^{n} \given q_{j}\geq p_{j} \text{ for } j\in J,\ q_{j}=p_{j} \text{ for } j\not\in J}.
\]

Given $J \subseteq \Set{1,\dotsc,n}$ with $J \neq \emptyset$ and $p \in T^{n}$, we define a set $\calQ_{J}(p) \subset T^{n}$ as follows.
If $\abs{J}=1$, then $\calQ_{J}(p)$ consists of the (unique) maximal element of $\upset_{J}p \setminus \calU$, if the latter set is nonempty, and is empty otherwise.
If $\abs{J} \geq 2$, then $\calQ_{J}(p)$ is a maximal set of maximal elements of $\upset_{J}p \setminus \calW_{n-\abs{J}+1}$ such that the sets $\upset_{J} q \setminus \calW_{n-\abs{J}+2}$ are pairwise disjoint for $q\in\calQ_{J}(p)$.

Then, recursively, let $\calR_{\emptyset}(p) := \Set{p}$,
\[
\calR_{J}(p) := \cup_{J' \subset J} \cup_{p' \in \calQ_{J}(p)} \calR_{J'}(p'),
\]
where the first union runs ovel all subsets of $J$ with cardinality $\abs{J'}=\abs{J}-1$.

We claim that, for every $p\in\upset\omega$ and every $J \subseteq \Set{1,\dotsc,n}$ with $J\neq\emptyset$, we have
\begin{equation}
\label{eq:boxes-cover}
\bigcup_{p' \in \calR_{J}(p)} \downset p'
\supseteq
\upset_{J} p \setminus \calW_{n-\abs{J}+1},
\end{equation}
where we set $\calW_{n} := \calU$ to simplify notation.
We prove \eqref{eq:boxes-cover} by induction on $\abs{J}$.
For $\abs{J}=1$, the claim \eqref{eq:boxes-cover} obviously holds.
Let now $J$ with $\abs{J}\geq 2$ be given, and suppose that \eqref{eq:boxes-cover} is known for all proper subsets of $J$.
Let
\[
\calD := \bigcup_{p' \in \calR_{J}(p)} \downset p',
\quad
\calP := \upset_{J} p \setminus \calW_{n-\abs{J}+1}.
\]
By the inductive hypothesis,
\begin{equation}
\label{eq:8}
\calD \supseteq \upset_{J'} p' \setminus \calW_{n-\abs{J}+2}
\end{equation}
for every $p'\in Q_{J}(p)$ and every $J' \subsetneq J$.
Suppose that
\begin{equation}
\label{eq:10}
\calD \not\supseteq \calP.
\end{equation}
Choose a maximal $q \in \calP\setminus\calD$.
Since $\calD$ is a down-set, $q$ is also a maximal element of $\calP$.
We claim that
\begin{equation}
\label{eq:9}
(\upset_{J} q \cap \upset_{J}p')\setminus\calW_{n-\abs{J}+2} = \emptyset
\text{ for all } p'\in\calQ_{J}(p).
\end{equation}
Indeed, suppose for a contradiction that there exists $q' \in (\upset_{J} q \cap \upset_{J}p')\setminus\calW_{n-\abs{J}+2}$, and let $q'$ be minimal with this property.
Since $\calW_{n-\abs{J}+2}$ is an up-set, $q'$ is also a minimal element of $\upset_{J} q \cap \upset_{J}p'$.
Since $q,p' \in \upset_{J}p$, $q'$ is in fact the coordinatewise maximum of $q,p'$.
Since $q$ and $p'$ are distinct maximal elements of $\calP$, in fact $q'$ coincides with $p'$ in at least one coordinate, so $q' \in \upset_{J'}p'$ for some $J' \subsetneq J$.
Now, \eqref{eq:8} implies that $q' \in \calD$, and, since $\calD$ is a down-set and $q' \geq q$, also $q\in\calD$, a contradiction.

Therefore, \eqref{eq:9} holds.
But this contradicts the maximality of $\calQ_{J}(p)$.
Thus the assumption \eqref{eq:10} is false, and we obtain \eqref{eq:boxes-cover}.

Let $p\geq \omega$.
For $2 \leq \abs{J} \leq n$, we have
\begin{align*}
1
&\gtrsim
\bfV^{\mu}(\omega)
\\ &\geq
\bfV^{\mu}(p)
\\ &\geq
\sum_{q\in\calQ_{J}(p)} \int_{\upset_{J} q \setminus \calW_{n-\abs{J}+2}} f
\\ &\geq
\sum_{q\in\calQ_{J}(p)} (\bfI f(q) - \bfI (f \one_{\calW_{n-\abs{J}+2}})(\omega))
\intertext{by definition of $q\in\calW_{n-\abs{J}+1}$ and by \eqref{eq:7},}
&\geq
\sum_{q\in\calQ_{J}(p)} (\epsilon_{n-\abs{J}+1} - (\epsilon_{n-\abs{J}+2}/\epsilon)^{1/2}/10)
\\ &\gtrsim
\abs{\calQ_{J}(p)} \epsilon_{n-\abs{J}+1}.
\end{align*}

It follows that
\[
\epsilon_{1}\dotsm\epsilon_{n-1}\abs{\calR_{\Set{1,\dotsc,n}}(\omega)} \lesssim 1.
\]
Hence, by \eqref{eq:boxes-cover},
\begin{align*}
\bfV^\mu(\om)-\bfV^\mu_\eps(\om)
&=
\int_{\upset \omega \setminus \calW_{1}} f
\\ &\leq
\sum_{p' \in \calR_{\Set{1,\dotsc,n}}(\omega)} \int_{\downset p'} f
\\ &=
\sum_{p' \in \calR_{\Set{1,\dotsc,n}}(\omega)} \bfV_{p'}^{\mu}(\omega)
\\ &\leq
\epsilon' \abs{\calR_{\Set{1,\dotsc,n}}(\omega)}
\\ &\lesssim
\frac{\epsilon'}{\epsilon_{1}\dotsm\epsilon_{n-1}}
= \epsilon.
\end{align*}
Therefore, by \eqref{eq:7},
\[
1/3
\leq
\bfV^\mu(\om)
=
(\bfV^\mu(\om)-\bfV^\mu_\eps(\om)) + \bfV^\mu_\eps(\om)
\leq
C\epsilon + 1/10.
\]
This inequality is false if $\epsilon$ is sufficiently small, contradicting the assumption $\calU \subseteq \calW_{n-1}$.
\end{proof}

\subsection{Box condition implies hereditary Carleson}
\label{boxHC}
We refer to \cite[Lemma 3.1]{AHMV} or \cite[Lemma 7.1]{AMPVZ} for the following lemma.
\begin{lemma}[Balancing lemma]
\label{lem:balance}
Let $\nu: T^n \to [0, \infty)$ with
\[
\calE[\nu]
=
\int \bfV^{\nu}\dif\nu
\geq
A \abs{\nu}.
\]
Then there exists a down-set $\tilde{E} \subset T^{n}$ such that for the measure $\tilde{\nu} := \nu\one_{\tilde{E}}$ we have
\[
\bfV^{\tilde{\nu}} \geq \frac{A}{3}
\quad \text{on } \tilde{E},
\]
and 
\[
\mathcal{E}[\tilde{\nu}] \geq \frac{1}{3} \mathcal{E}[\nu].
\]
\end{lemma}

The next result contains the last missing inequality in Theorem~\ref{thm:main}.
\begin{theorem}
\label{thm:box=>HC}
Let $n\geq 2$.
Let $w:T^n\to [0, \infty)$ be such that the surrogate maximal principle \eqref{eq:max-surrogate} holds.
Then, for every $\nu: T^n \to [0, \infty)$, we have
\[
[w,\nu]_{HC} \lesssim [w,\nu]_{Box}.
\]
\end{theorem}
\begin{proof}
By scaling, we may assume $[w,\nu]_{Box}=1$ without loss of generality.
Let $A := [w,\nu]_{HC}$.
Let $E \subset T^{2}$ be a subset such that $\mu=\nu\one_{E} \neq 0$ and $\mathcal{E}[\mu] = A \abs{\mu}$ (such a subset exists because we assume that $T^{n}$ is finite).
By Lemma~\ref{lem:balance}, there exists a further subset $\tilde{E} \subset T^{2}$ such that $\tilde{\mu} := \mu \one_{\tilde{E}}$ satisfies
\[
\bfV^{\tilde{\mu}} \geq \frac{A}{3} \text{ on } \tilde{E}
\]
and $\tilde\mu \neq 0$.
Thus, replacing $\mu$ by $\tilde{\mu}$, we may assume $\bfV^{\mu} \geq A/3$ on $\supp \mu$.

By Lemma~\ref{2main} applied with $\mu/A$ in place of $\mu$, for sufficiently small $\epsilon,\theta>0$, we have
\begin{equation}
\label{eq:good-energy-dominates}
\int \bfV^{\mu}_{\epsilon A,good} \dif\mu
\geq
2 \theta \calE[\mu].
\end{equation}
We claim that, with these values of $\epsilon$ and $\theta$, we have
\begin{equation}
\label{eq:good-energy-dominates2}
\mathcal{E}[\mu]
\leq \frac{\theta}{1-\theta}
\sum_{\alpha : \theta\epsilon A \bfI^{*}\mu(\alpha) \leq \mathcal{E}_{\alpha}[\mu]} w(\alpha) (\bfI^{*} \mu(\alpha))^{2}.
\end{equation}
Indeed, suppose that $\alpha$ is such that
\[
\theta\epsilon A \bfI^{*}\mu(\alpha)
>
\mathcal{E}_{\alpha}[\mu]
=
\sum_{\omega \leq \alpha} \mu(\omega) \bfV^{\mu}_{\alpha}(\omega),
\quad
\bfV^{\mu}_{\alpha}(\omega) = \sum_{\beta : \omega \leq \beta \leq \alpha} w(\beta) (\bfI^{*}\mu)(\beta),
\]
where the latter definition is from \eqref{eq:V-interval}.
Then we have
\begin{align*}
\sum_{\omega \leq \alpha : \bfV^{\mu}_{\alpha}(\omega) \leq \epsilon A} \mu(\omega)
&=
\bfI^{*}\mu(\alpha) - \sum_{\omega \leq \alpha : \bfV^{\mu}_{\alpha}(\omega) > \epsilon A} \mu(\omega)
\\ &\geq
\bfI^{*}\mu(\alpha) - \frac{1}{\epsilon A} \sum_{\omega \leq \alpha} \bfV^{\mu}_{\alpha}(\omega) \mu(\omega)
\\ &\geq
(1-\theta) \bfI^{*}\mu(\alpha).
\end{align*}
It follows that
\begin{align*}
\sum_{\alpha : \theta\epsilon A \bfI^{*}\mu(\alpha) > \mathcal{E}_{\alpha}[\mu]} w(\alpha) (\bfI^{*} \mu(\alpha))^{2}
&\leq
\sum_{\alpha} w(\alpha) \bfI^{*} \mu(\alpha) \frac{1}{1-\theta} \sum_{\omega \leq \alpha : \bfV^{\mu}_{\alpha}(\omega) \leq \epsilon A} \mu(\omega)
\\ &=
\frac{1}{1-\theta} \sum_{\omega} \mu(\omega) \sum_{\alpha \geq \omega : \bfV^{\mu}_{\alpha}(\omega) \leq \epsilon A} w(\alpha) \bfI^{*} \mu(\alpha)
\\ &=
\frac{1}{1-\theta} \sum_{\omega} \mu(\omega) (\bfV^{\mu}-\bfV^{\mu}_{good,\epsilon A})(\omega)
\\ &\leq
\frac{1-2\theta}{1-\theta} \calE[\mu].
\end{align*}
This implies the claim \eqref{eq:good-energy-dominates2}.

By Lemma~\ref{partialV} again, and since $\bfV^{\mu} \geq A/4$ on $\supp \mu$, we also have
\begin{equation}
\label{eq:small-V-energy}
\mathcal{E}_{c' A}[\mu]
\lesssim
(c'A)^{\kappa} \abs{\mu}^{\kappa} \mathcal{E}[\mu]^{1-\kappa}
\lesssim
(c')^{\kappa} \mathcal{E}[\mu].
\end{equation}
Taking $c'$ sufficiently small and combining \eqref{eq:small-V-energy} with \eqref{eq:good-energy-dominates2}, we obtain
\[
\mathcal{E}[\mu]
\lesssim
\sum_{\alpha \in \mathcal{R}} w(\alpha) (\bfI^{*}\mu(\alpha))^{2},
\quad
\mathcal{R} := \Set{ \alpha\in T^{n} \given \theta\epsilon A \bfI^{*}\mu(\alpha) \leq \mathcal{E}_{\alpha}[\mu], \bfV^{\mu}(\alpha) \geq c' A }.
\]
For each $\alpha \in \mathcal{R}$, we have
\[
\theta\epsilon A \bfI^{*}\mu(\alpha)
\leq
\mathcal{E}_{\alpha}[\mu]
\leq
\mathcal{E}_{\alpha}[\nu]
\leq
[w,\nu]_{Box} \bfI^{*}\nu(\alpha)
=
\bfI^{*}\sigma(\alpha),
\]
where $\sigma := \nu \one_{F}$, $F := \Set{ \beta \in T^{n} \given \exists \alpha\in\mathcal{R}, \alpha\geq\beta }$.
It follows that
\begin{equation}
\label{eq:2}
A^{2} \mathcal{E}[\mu] \lesssim \mathcal{E}[\sigma].
\end{equation}
On the other hand, using the definition of $A$, the fact that $\bfV^{\mu} \gtrsim A$ on $\supp\sigma$, and the Cauchy--Schwarz inequality, we obtain
\begin{equation}
\label{eq:3}
\mathcal{E}[\sigma]
\leq
A \abs{\sigma}
\lesssim
\int \bfV^{\mu} \dif\sigma
\leq
\mathcal{E}[\mu]^{1/2} \mathcal{E}[\sigma]^{1/2}.
\end{equation}
From \eqref{eq:3}, we obtain $\mathcal{E}[\sigma] \lesssim \mathcal{E}[\mu]$, and inserting this into \eqref{eq:2} gives $A \lesssim 1$.
\end{proof}

\section{What we cannot prove}
\label{cannot}

The main problem with pushing the results to $n$-trees, $n\ge 4$, lies with Lemma \ref{SmEMaj2} and Lemma \ref{SmEMaj3}.
Let us start with majorization on a simple dyadic tree.
All trees below are big but finite.
Let $f$, $g$ be two non-negative functions on a simple dyadic tree $T$.
As always $If(v)$ means summing $f(u)$  ``up" from $v$ to root $o$.

Here is the analog of Lemma \ref{SmEMaj2} and Lemma \ref{SmEMaj3}. The big difference of the lemma below is that it involves two functions: $f, g$. This is not the case for Lemma \ref{SmEMaj2} and Lemma \ref{SmEMaj3} that involve one function.

\begin{lemma} 
\label{ET}
Let $\supp f \subset \{Ig \le \delta\}$. Let $g$ be a superadditive function.
There exists $\vf:T\to \bR_+$ such that
\begin{equation}
\label{belowT}
a) \,\,I\vf(\om) \ge If(\om) \quad \forall \om\in \pd T\colon Ig(\om)\in [\la, 2\la]
\end{equation}
\begin{equation}
\label{aboveT}
b)\,\,\int_T\vf^2 \le C\frac{\delta}{\la} \int_T f^2.
\end{equation}
\end{lemma}
\begin{proof}
Put
$$
\vf = \la^{-1}I f\cdot g\cdot {\bf 1}_{I g \le 4\la}\,,
$$
and see \cite{AMPVZ}.
\end{proof}

Now let us see what happens on bi-tree $T^2$.
As before $\bI f(v)$ means summing $f(u)$  ``up"  over all ancestors of $v$ from $v$ to root $o$. Notice that now a vertex may have two parents.

\begin{conjecture} 
\label{ETc}
Let $\supp f \subset \{\bI g \le \delta\}$. Let $g$ be a function  superadditive in its both variables separately .
There exists $\vf:T^2\to \bR_+$ such that
\begin{equation}
\label{belowTc}
a) \,\,\bI\vf(\om) \ge \bI f(\om) \quad \forall \om\in \pd T^2\colon \bI g(\om)\in [\la, 2\la]
\end{equation}
\begin{equation}
\label{aboveTc}
b)\,\,\int_{T^2}\vf^2 \le C\Big(\frac{\delta}{\la} \Big)^\tau\int_{T^2} f^2
\end{equation}
with some positive $\tau$.
\end{conjecture}
By analogy with the previous section one may think that given $f, g$ on $T^2$, such that
\begin{equation}
\label{supportT2}
\supp f \subset \{ \bI g\le \delta\}
\end{equation}
and having $g$ (super)additive on $T^2$, one constructs $\vf$ as in Lemma  \ref{ET} by formula
$$
\vf =\la^{-1} \bI f \cdot g\cdot {\bf 1}_{\bI g \le 4\la}\,.
$$

However this is false.  

What is true is the following: let   $\supp f\subset \{ \bI g\le \delta\}$ and  let $\la \ge 10\delta$, $\vf:= \la^{-1}(I_1 f \cdot I_2g +  I_1 g \cdot I_2 f + g\cdot \bI f )$. Then
\begin{equation}
\label{maj2}
\bI (\one_{\bI g \le 2\la} \cdot \vf) \ge  \bI f, \quad \text{where} \,\, \bI g \in [\la, 2\la]\,.
\end{equation}

So a) from the previous lemma can be generalized to bi-tree with the following formula for $\vf$:  
\begin{equation}
\label{vf2}
\vf= \la^{-1}(I_1 f \cdot I_2g +  I_1 g \cdot I_2 f + g\cdot \bI f )\cdot \one_{\bI g \le 2\la}.
\end{equation}

\bigskip

The main difficulty in generalizing Lemma \ref{ET} to bi-trees is that we cannot prove b) of this lemma on bi-tree. This is because we have no good estimate of $\int_{T^2} (\bI f)^2 g^2$ via $\int_{T^2} f^2$ for $g$ that are superadditive in both variables.

Notice that this hurdle is removed if $f=g$ because then 
$$
\bI (\la^{-1} g \bI f)= \bI (\la^{-1} f \bI f)\le \frac{\delta}{\la} \bI f \le \frac1{10} \bI f,
$$
and we have another $\vf$ for majorization: $\tilde \vf :=  c\la^{-1}(2I_1 f \cdot I_2f  )$, where $c=\frac{10}{9}$. In fact from \eqref{maj2} it now follows that
\begin{equation}
\label{maj2f}
\bI (\one_{\bI f \le 2\la} \cdot \tilde\vf) \ge  \bI f, \quad \text{where} \,\, \bI f \in [\la, 2\la]\,.
\end{equation}
The analog of inequality b) of Lemma \ref{ET} $\equiv$ \eqref{aboveT} on bi-tree
now follows from Lemma \ref{lem:en2:2}.

\bigskip

For tri-tree we do not have the analog of Lemma \ref{ET}  with two functions $f, g$, as we do not have it even on bi-tree.

But similarly with \eqref{vf2} we can put
\begin{equation}
\label{vf3}
\begin{split}
&\vf= \la^{-1}(I_1 f \cdot \bI_{23} g + I_2 f \cdot \bI_{13} g  +I_3 f \cdot \bI_{12} g  + 
\\
&I_1 g f \cdot \bI_{23} f + I_2 g \cdot \bI_{13} f  +I_3 g \cdot \bI_{12} f + g \bfI f) .
\end{split}
\end{equation}
Again this function $\vf$ will satisfy
\begin{equation}
\label{maj3}
\bfI (\cdot \one_{\bfI \le 2\la}\cdot \vf) \ge \bfI f, \quad\text{where} \,\, \bfI g \in [\la, 2\la],
\end{equation}
which the analog of a) of Lemma \ref{ET} (and the analog of \eqref{maj2}).
However, we cannot prove the analog  of  b) of Lemma \ref{ET}  for this function.

The main difficulty in generalizing Lemma \ref{ET} to tri-trees is that we cannot prove b) of this lemma on tri-tree. This is because we have no good estimate of $\int_{T^3} (\bfI f)^2 g^2$ via $\int_{T^3} f^2$ for $g$ that are superadditive in both variables.

Notice that this hurdle is removed if $f=g$ because then 
$$
\bfI (\la^{-1} g \bfI f)= \bfI (\la^{-1} f \bfI f)\le \frac{\delta}{\la} \bfI f \le \frac1{10} \bfI f,
$$
and in place of $\vf$ from \eqref{vf3}, we have another $\vf$ for majorization: 
$$
\tilde \vf :=  c\la^{-1}(2I_1 f \cdot \bI_{23} f + 2I_2 f \cdot \bI_{13} f  +2I_3 f \cdot \bI_{12} f ),
$$
 where $c=\frac{10}{9}$. In fact from \eqref{maj3} it now follows that
\begin{equation}
\label{maj3f}
\bI (\one_{\bfI f \le 2\la} \cdot \tilde\vf) \ge  \bfI f, \quad \text{where} \,\, \bfI f \in [\la, 2\la]\,.
\end{equation}
The analog of inequality b)  of Lemma \ref{ET} $\equiv$ \eqref{aboveT} on tri-tree
now follows from Lemma \ref{lem:energy-bd:3}.

\bigskip

\subsection{What goes wrong on $4$-tree}
The reader has the right to ask: you do not know how to estimate $\int_{T^2} (\bI f)^2 g^2$ via $\int_{T^2} f^2$ and you do not know  how to estimate $\int_{T^3} (\bfI f)^2 g^2$ via $\int_{T^2} f^2$, but you know how to remove this hurdle in the case $f=g$. May be one can also remove this hurdle for $f=g$ on $d$-tree, $d\ge 4$?

Unfortunately, we can see  now that the trick does not work for $d\ge 4$. Let us notice that by the analogy with \eqref{vf2}, \eqref{vf3} we can construct $\vf$ for $4$-tree:
\begin{equation}
\label{vf4}
\begin{split}
&\vf= \la^{-1}(I_1 f \cdot \bI_{234} g + I_2 f \cdot \bI_{134} g  +I_3 f \cdot \bI_{124} g  + I_4 f \cdot \bI_{123} g +
\\
&I_1 g  \cdot \bI_{234} f + I_2 g \cdot \bI_{134} f  +I_3 g \cdot \bI_{124} f +I_4 g \cdot \bI_{123} f  +
\\
& \bI_{12} g \cdot \bI_{34} f +   \bI_{23} g  \cdot \bI_{14} f +   \bI_{34} g  \cdot \bI_{12} f +   \bI_{12}  f \cdot \bI_{34} g+   \bI_{23}  f \cdot \bI_{14} g +   \bI_{34}  f \cdot \bI_{12} g+
\\
& g \bfI f) .
\end{split}
\end{equation}
Here $\bfI$ means summation in all $4$ variables, the Hardy operator on $T^4$.
Let us consider what happens for the case $g=f$. We again can absorb the last term $g \bfI f= f \bfI f \le \delta f$ into the left hand side because $\supp f \subset \{\bfI f \le \delta\}$.

But to prove the analog of b) of Lemma \ref{ET} we would need to know how to estimate e. g.
$$
\int_{T^4} (\bI_{12} f \cdot \bI_{34} f)^2 \le C\int_{T^4} f^2\,.
$$
We do not know  how to achieve such an estimate.

To feel this difficulty better, let us prove Lemma \ref{ET}, where the main point is the following ``weighted'' estimate of 
\begin{equation}
\label{weightedT}
\supp f \subset \{ I g\le \delta\}\Rightarrow\int_T (I f)^2 g^2 \le C\delta\|Ig\|_\infty \int_T f^2\quad\text{for superadditive}\,\, g\,.
\end{equation}

\subsubsection{The proof of Lemma \ref{ET} and the explanation where the proof breaks down on bi-tree}
\label{prET}

We just repeat  the proof from \cite{AMPVZ}, but we emphasize why the proof does not work for very similar estimate of $\int_{T^2 }(\bI f)^2 g^2$.
We are in the assumptions of Lemma \ref{ET}. That is, we are given two functions $f, g$ on tree $T$, and 
$$
1)\,\, \supp f \subset \{Ig \le \delta\},
$$
$$
2)\,\,g  \,\, \text{is a superadditive function.}
$$
We need to see why the key estimate \eqref{weightedT} works on $T$ and will not work on $T^2$ if one replaces $I$ by $\bI$ and $T$ by $T^2$ everywhere.

We start with lemma that holds regardless of operator and medium.

\begin{lemma}
\label{lem:I2-positive}
Let $K$ be an integral operator with a positive kernel and $f,g$ positive functions.
Then
\[
\int (Kf)^{2} g \leq \Bigl( \sup_{\supp g} KK^{*}g \Bigr) \int f^{2}.
\]
\end{lemma}
\begin{proof}
Without loss of generality $f$ is positive.
By duality we have
\[
\int (Kf)^{2} g
=
\int f K^{*}(Kf \cdot g)
\leq
\norm{f}_{2} \norm{K^{*}(Kf \cdot g)}_{2}.
\]
We call the operator and its kernel by the same letter $K$. By the hypothesis $Kh(x) = \int K(x,y) h(y)$ with a positive kernel $K$.
Hence
\begin{align*}
\norm{K^{*}(Kf \cdot g)}_{2}^{2}
&=
\int K^{*}(Kf \cdot g) K^{*}(Kf \cdot g)
\\ &=
\int K(x,y) ((If)(x) g(x)) K(x',y) ((Kf)(x') g(x')) \dif(x,x',y)
\\ &\leq
\int \frac12 (Kf(x)^{2}+Kf(x')^{2}) K(x,y) (g(x)) K(x',y) (g(x')) \dif(x,x',y)
\\ &=
\frac12 \int K^{*}((Kf)^{2} \cdot g) K^{*}(g) + \int K^{*}(g) K^{*}((Kf)^{2} \cdot g)
\\ &=
\int (KK^{*}g) \cdot (Kf)^{2} \cdot g
\\ &\leq
\Bigl( \sup_{\supp g} KK^{*}g \Bigr) \int (Kf)^{2} \cdot g.
\end{align*}
Substituting the second displayed estimate into the first we obtain
\[
\int (Kf)^{2} g
\leq
\norm{f}_{2} \Bigl( \sup_{\supp g} KK^{*}g \Bigr) \Bigl( \int (Kf)^{2} \cdot g \Bigr)^{1/2}.
\]
The conclusion follows.
\end{proof}

In the preceding lemma operator $K$ could have been either $I$ on $T$ or $\bI$ on $T^2$, this did not matter. But in the next lemma, it matters whether we are on $T$ or $T^2$.

\begin{lemma}
\label{lem:supadditive-l1linf}
Let $T$ be a finite tree and $g,h : T \to [0,\infty)$.
Assume that $g$ is superadditive and $\la=\|Ih\|_{L^\infty(\supp g)}$.
Then for every $\beta \in T$ we have
\[
I(gh)(\beta)=\sum_{\alpha \leq \beta} g(\alpha) h(\alpha)
\leq
\lambda g(\beta).
\]
\end{lemma}
\begin{proof}
Without loss of generality we may consider the case when $\beta$ is the unique maximal element of $T$ and $T= \supp g$.
We induct on the depth of the tree.
Let $T$ be given and suppose that the claim is known for all its branches.
Then by the inductive hypothesis and superadditivity of $g$ we have
\begin{align*}
\sum_{\alpha \leq \beta} g(\alpha) h(\alpha)
&=
g(\beta) h(\beta) + \sum_{\beta' \in \ch(\beta)} \sum_{\alpha \leq \beta'} g(\alpha) h(\alpha)
\\ &\leq
g(\beta) h(\beta) + \sum_{\beta' \in \ch(\beta)} g(\beta') \sup_{\alpha \leq \beta'} \sum_{\alpha \leq \alpha' \leq \beta'} h(\alpha')
\\ &\leq
g(\beta) h(\beta) + \sum_{\beta' \in \ch(\beta)} g(\beta') \sup_{\alpha < \beta} \sum_{\alpha \leq \alpha' < \beta} h(\alpha')
\\ &\leq^{key}
g(\beta) h(\beta) + g(\beta) \sup_{\alpha < \beta} \sum_{\alpha \leq \alpha' < \beta} h(\alpha')
\\ &=
g(\beta) \sup_{\alpha \leq \beta} \sum_{\alpha \leq \alpha' \leq \beta} h(\alpha').
\qedhere
\end{align*}
\end{proof}
\begin{remark}
\label{FirstDifficulty}
It seems like this claim fails to be true on $T^2$. At least the reasoning fails. In Conjecture \ref{ETc} we had to assume that $g$ is superadditive in its both variables. This assumption is indispensable for us, because in our applications  of such a lemma on $T^2$ function $g$ on $T^2$ always comes from  some  function (measure)  $f$  additive on $T^3$ in each of its three variables. Function $g$ is always defined by a simple rule 
$g= I_i f \cdot \one_{\bfI f\le t}$, $i=1$ or $2$ or $3$. But such function $g$ is automatically superadditive in each of its two variables.

But if $g$ is superadditive in its both variables then the key estimate in the above lemma does not work. In fact, instead of having $\sum_{\beta' \in \ch(\beta)} g(\beta') \le g(\beta)$ we will have to write
$$
\sum_{\beta' \in \ch(\beta)} g(\beta') \le 2 g(\beta)\,.
$$
This seemingly innocuous change leads to accumulation of constant in the above proof, the proof breaks down if it cannot keep constant $1$ at every stage of induction.
\end{remark}
Now we present the proof of Lemma \ref{ET} by means of Lemma \ref{lem:I2-positive} and Lemma \ref{lem:supadditive-l1linf}.
Let $\vf = 2\la^{-1} If \cdot g \cdot \one_{Ig \le 4\la}$. Let $\om $ be such that $Ig (\om) \ge \la$. Then $f(\om)=0$ and $f(\gamma)=0$ for all ancestors of $\om$ up to the first $\gamma'$ such that $Ig (\gamma') \le \delta$. Hence,  on such $\om$
\[
\sum_{\gamma\ge \om}If \cdot g \cdot \one_{Ig \le 4\la} = \sum_{\gamma\ge \om}If \cdot g =If (\om) (\la -\delta) \ge \frac{\la}{2} If (\om).
\]
We checked \eqref{belowT} of Lemma \ref{ET}.

To check \ref{aboveT} we first apply Lemma \ref{lem:I2-positive} with
$$
K:= I\circ \one_{I g\le \delta},
$$
which a composition of multiplication operator and $I$.
Then
$$
\int_T \vf^2 = \frac{4}{\la^2}\int_T (If)^2 (g\one_{I g\le 4\la})^2 \le \frac{4}{\la^2}\sup_{\supp g} KK^* (g^2\one_{I g\le 4\la}) \int_T f^2\,.
$$
To understand  $\sup_{\supp g} KK^* (g^2\one_{I g\le 4\la}) $ we use Lemma \ref{lem:supadditive-l1linf}. By this lemma for any node $\al$
$$
K^* (g^2\one_{I g\le 4\la})(\al) \le I^* (g^2\one_{I g\le 4\la})(\al) \le 4\la g(\al)\,.
$$
Now we are left to estimate $Kg= I (\one_{I g \le \delta} g)$. But just by definition of $I$ we have
\begin{equation}
\label{Icut}
 I (\one_{I g \le \delta} g) \le  \delta.
 \end{equation}
 So $\sup_{\supp g} KK^* (g^2\one_{I g\le 4\la}) \le 4\delta\la$ and we get 
 $$
 \int_T \vf^2 \le \frac{16 \delta}{\la} \int f^2\,.
 $$
 \begin{remark}
 \label{SecondDifficulty}
 We already observed one obstacle to prove Conjecture \ref{ETc}. We did this in Remark \ref{FirstDifficulty}. Now let us observe, that even if we would manage to overcome this first difficulty mention in that remark, we still have another very serious one: the analog of inequality \eqref{Icut} is blatantly false on $T^2$. The fallowing inequality is generically false:
 \begin{equation}
\label{bIcut}
 \bI (\one_{\bI g \le \delta} g) \le  \delta.
 \end{equation}
 \end{remark}

\begin{remark}
We feel that if we would know how to prove Conjecture \ref{ETc} on $T^2$ we would be able to prove it on any $T^d$. This would prove our surrogate maximal principle in any dimension. This, in its turn, would characterize embedding measures on graphs $T^d$ not only for $d=1, 2, 3$, but for arbitrary $d$.
\end{remark}

\end{document}